\documentclass[journal]{IEEEtran}

\IEEEoverridecommandlockouts                              
\bstctlcite{IEEEexample:BSTcontrol}                                                          

\usepackage{flushend} 
\usepackage{romannum}
\usepackage[usenames,dvipsnames]{xcolor}
\usepackage{tkz-berge}
\usetikzlibrary{fit,shapes}                                     
\usepackage{tikz}
\usepackage{calrsfs}
\DeclareMathAlphabet{\pazocal}{OMS}{zplm}{m}{n}
\usepackage{graphicx}
\usepackage{graphics} 
\usepackage{epsfig} 
\usepackage{mathptmx}
\usepackage{subcaption}
\usepackage{times} 
\usepackage{tikz}
\usetikzlibrary{positioning}
\usepackage{amsmath, amssymb}
\usepackage{amsthm}
\usepackage{amsfonts} 
\usepackage{setspace}
\usepackage{multirow}
\usepackage{textcomp}
\usepackage[english]{babel}
\usepackage{float} 
\usepackage{algorithmicx}
\usepackage[section]{algorithm}
\usepackage{algpascal}
\usepackage{algc}
\usepackage{algcompatible}
\usepackage{algpseudocode}
\let\oldReturn\Return
\renewcommand{\Return}{\State\oldReturn}
\usepackage{mathtools}
\usepackage{bm}
\usepackage{mathptmx}
\usepackage{times} 
\usepackage{mathtools, cuted}
\usepackage{textcomp}
\usepackage[english]{babel}
\usepackage{rotating}
\usepackage{pgfplots}
\pgfplotsset{compat=1.5}
\usepackage{pgfplotstable}
\usepackage{filecontents} 
\usepackage{makecell}
\usepackage{diagbox}

\newcommand{\B}{\pazocal{B}}
\newcommand{\C}{\pazocal{C}}
\newcommand{\K}{\pazocal{K}}

\newcommand{\D}{\pazocal{D}}
\newcommand{\E}{\pazocal{E}}

\newcommand{\U}{\pazocal{U}}

\newcommand{\J}{\pazocal{J}}
\newcommand{\I}{\pazocal{I}}
\newcommand{\sI}{\pazocal{\scriptscriptstyle I}}

\renewcommand{\P}{\pazocal{P}}

\renewcommand{\S}{\pazocal{S}}
\newcommand{\x}{{X}}
\renewcommand{\u}{{U}}

\newcommand{\remove}[1]{}
\newcommand{\G}{\scriptscriptstyle{G}}

\def \*{\star}
\def \10n{\!\!\!\!\!\!\!\!\!\!}

\newcommand{\bK}{\bar{K}}

\newcommand{\bk}{\bar{K}}
\newcommand{\R}{\mathbb{R}}
\newcommand{\bA}{\bar{A}}
\newcommand{\bB}{\bar{B}}
\newcommand{\bC}{\bar{C}}

\newcommand{\tV}{\widetilde{V}}
\newcommand{\tv}{\tilde{v}}

\renewcommand{\baselinestretch}{0.96}  
\renewcommand{\thesection}{\arabic{section}}

\graphicspath{{Figures/}}

\makeatletter
\newcommand{\ssymbol}[1]{^{\@fnsymbol{#1}}}
\newcommand{\specificthanks}[1]{\@fnsymbol{#1}}
\makeatother

\newtheorem{theorem}{Theorem}[section]
\newtheorem{lemma}[theorem]{Lemma}
\newtheorem{cor}[theorem]{Corollary}
\newtheorem{proposition}[theorem]{Proposition}

\theoremstyle{definition}
\newtheorem{definition}[theorem]{Definition}

\newtheorem{problem}[theorem]{Problem}

\theoremstyle{remark}

\newtheorem{remark}{Remark}
\newtheorem{assume}{Assumption}

\title{\LARGE \bf Verification and Design of Resilient Closed-Loop Structured System}
\author{RaviTeja Gundeti,
		Shana~Moothedath
        and~Prasanna~Chaporkar
\thanks{The authors are in the Department of Electrical Engineering, Indian Institute of Technology Bombay, India. Email: ravi5gundeti@gmail.com, shana@ee.iitb.ac.in, chaporkar@ee.iitb.ac.in.}}

\begin{document}
\maketitle
\thispagestyle{empty}
\pagestyle{empty}
\begin{abstract}
This paper addresses resilience of large-scale closed-loop structured  systems in the sense of arbitrary pole placement when subject to failure of feedback links. Given a  structured system with input, output,  and feedback matrices, we first aim to verify whether the closed-loop structured system is resilient to simultaneous failure of {\em any} subset of feedback links of cardinality at most $\gamma$. Subsequently, we address the associated design problem in which given a  structured system with input and output matrices,  we need to design a sparsest feedback matrix that ensures the resilience of the resulting closed-loop structured system to simultaneous failure of at most {\em any} $\gamma$ feedback links. We first prove that the verification problem is NP-complete even for {\em  irreducible} systems and  the design problem is NP-hard even for so-called {\em structurally cyclic} systems. We also show that the design problem is  inapproximable to factor $(1 - o(1)) \log n$, where $n$ denotes the system dimension. Then we propose algorithms to solve  both the problems: a pseudo-polynomial algorithm to address the verification problem of irreducible systems  and a polynomial-time  $O(\log n)$-optimal approximation algorithm to solve the design problem for a special feedback structure, so-called {\em back-edge} feedback structure.
\end{abstract}
\begin{IEEEkeywords}
Resilient systems, structured linear time-invariant systems, arbitrary pole placement, optimal feedback design
\end{IEEEkeywords}
\section{Introduction}\label{sec:intro}
Complex networks and cyber-physical systems have applications in a wide variety of areas including multi-agent networks, power networks, social communication networks, biological networks, and distribution networks \cite{LiuSloBar:11}.  Most of these networks are well represented as linear time-invariant (LTI) dynamical systems.    In LTI systems with output feedback, the feedback matrix decides which output to be fed as feedback to which input and what control actions to be taken. Feedback selection for decentralized control in LTI systems is a fundamental problem in control theory.  Feedback selection aims at  designing a feedback matrix such that the closed-loop system satisfies arbitrary pole placement property and thus guarantees any desired closed-loop performance.

  Complex networks often  consist of interconnected components  with spatially distributed actuators and sensors. Establishing  feedback connections among spatially distributed actuators and sensors that are resilient to failure and attacks is difficult.   In many complex networks, including power networks and distribution networks, some of the feedback links become dysfunctional over time due to the vulnerability of the actuation, sensing and feedback mechanism. Additionally, many times there are targeted disruptive attacks by adversaries which tampers the structure of the network. Since the structure of the network is endogenous in nature, these changes affect the properties of the network, and the properties affect the system's performance.  In order to guarantee any desired performance of the closed-loop system, it is essential that the  feedback matrix is robust/resilient to disruptive scenarios such as natural failures or attacks by skilled and intelligent, adversarial  agents \cite{CarAmiSas:08}, \cite{iliXieKhaMou:10}. 
 
 Moreover, real-world networks are of large system dimension and complex graph pattern, and hence  most of the entries of the system matrices are not known precisely.  {\em Structural analysis} is a framework that is used to analyze the properties of LTI systems when only the sparsity patterns of the system matrices are known \cite{Rei:88}. Structural analysis performs control theoretic analysis of systems using the sparsity pattern, i.e., the zero/non-zero pattern, of the system matrices. 
The strength of structural analysis is that most of the structural properties, like  structural controllability, structural observability, and pole placement, of structured systems, are `generic' in nature \cite{Rei:88}, \cite{ComDioWou:02}. Hence, if the sparsity pattern of a system satisfies these properties, then `almost all' systems with the same sparsity pattern satisfy the analogous control-theoretic properties.
There are graph-theoretic conditions to verify the control-theoretic properties of the structured system. However, there are no known criteria to verify resilience of a system or design resilient systems. 

Often cyber-physical systems like power networks and water distribution networks undergo failure of interaction links in the system matrices due to aging and/or attacks by intelligent adversaries that tamper the structure of the system. Verification and design of resilient feedback matrix are critical to guarantee the desired operating condition of the system during such adversarial situations. Developing computationally efficient algorithms to verify and design resiliency of complex cyber-physical systems is the key focus of this paper.

In this paper,  we consider the resilience of the closed-loop structured system towards maintaining arbitrary pole placement property and focus on two problems.
Given a  structured system with state, input, output,  and  feedback matrices, we first aim to verify whether the closed-loop system is resilient to simultaneous failure of {\em any} $\gamma$ feedback links. The set of feedback links that undergo failure can be any arbitrary set of cardinality at most $\gamma$, since in real-world systems the connections that undergo attacks or failure is unknown {\em a priori}.  At present, there is no computationally efficient algorithm  to verify resilience of a closed-loop structured system when {\em any} subset of feedback links with cardinality bounded by a specified number can fail.  The exhaustive search-based algorithm requires verifying the arbitrary pole placement property for failure of all possible combinations of feedback links of cardinality $\gamma$ or less, which is exponential number of cases.  Then, we address the associated design problem, in which we need to design a sparsest feedback matrix that ensures the resilience of the closed-loop structured system to simultaneous failure of any subset of feedback links of cardinality at most  $\gamma$.  The key contributions of this paper are as follows:
\begin{enumerate}
\item[$\bullet$] We prove that, given structured state, input, output, and  feedback matrices, verifying resilience of the closed-loop structured system towards maintaining arbitrary pole placement property subject  to failure of any subset of feedback links whose cardinality is at most $\gamma$ is NP-complete  (Theorem~\ref{th:NP}). We prove that even for irreducible\footnote{A directed graph is said to be irreducible, if there exists a directed path between any two arbitrary vertices of it.} systems, verifying resilience of the closed-loop structured system subject to failure of any subset of feedback links whose cardinality is at most $\gamma$ is NP-complete (Corollary~\ref{cor:NP_irreducible}). 

\item[$\bullet$] We prove that, given structured state, input, and output matrices, designing a sparsest feedback matrix such that the resulting closed-loop system is resilient to failure of any subset of feedback links of cardinality at most $\gamma$ is NP-hard (Theorem~\ref{th:prob1_NP}).  We also show  that the design problem is inapproximable to factor $(1-o(1))\log\,n$, where $n$ denotes the system dimension (Theorem~\ref{th:approxi_NP_sol}).  We show that the NP-hardness and the inapproximability results of the design problem hold even for a widely practical subclass of systems, known as {\em structurally cyclic} systems, the class of systems in which all state nodes are spanned by a disjoint set of cycles\footnote{In a directed graph, a cycle is a directed closed walk with no repetitions of vertices and edges, except the starting and ending vertex.}.

\item[$\bullet$] We provide a polynomial-time approximation algorithm of approximation factor $O(\log\,n)$  for the sparsest resilient feedback design problem (Theorem~\ref{th:th_multi_approx}) for  structurally cyclic systems with a  special feedback structure, so-called {\em back-edge} feedback structure.  We show that the design problem is NP-hard and inapproximable to factor $(1-o(1))\log\,n$ for this class of systems, and hence the algorithm is an an order-optimal polynomial-time approximation algorithm.
\item[$\bullet$] We present polynomial time algorithms to verify the resilience of feedback matrix for $\gamma=1$ and $\gamma=2$ (Algorithms~\ref{alg:k=1} and~\ref{alg:k=2}), and prove the correctness and complexity of the algorithms (Theorems~\ref{th:k=1} and~\ref{th:k=2}).
 Then we extend these algorithms for one edge ($\gamma=1$) failure and two edges ($\gamma=2$) failures to a general case and prove its correctness and show that the complexity is pseudo-polynomial  with factor $\gamma$ (Theorem~\ref{th:algo_k}). Our algorithm performs computationally much better than exhaustive search-based algorithm and is computationally more efficient for small values of $\gamma$.
\end{enumerate}

The organization of this paper is as follows: Section~\ref{sec:prob} presents the formal description of  feedback resilience
verification problem and sparsest resilient feedback design problem. Section~\ref{sec:graph} discusses notations, few  preliminaries and some existing results used in the sequel. Section~\ref{sec:NP} analyzes the complexity of both problems and proves NP-completeness of feedback resilience
verification problem and NP-hardness of sparsest resilient feedback design problem. Section~\ref{sec:algo_design} presents an approximation algorithm for solving the sparsest resilient feedback design problem for structurally cyclic systems with a special feedback structure. Section~\ref{sec:algo_verify} presents a pseudo-polynomial  algorithm for solving the feedback resilience  verification problem for irreducible systems. Section~\ref{sec:conclu} gives the final concluding remarks.

\subsection{Related Work}\label{sec:rel}
  Resilience or robustness of complex networks subject to structural perturbations is of interest for a long time \cite{PicSezSil:81}. For instance, the robustness of  structured systems towards maintaining  structural controllability is addressed in \cite{RahAgh:13} by characterizing the role of nodes and links of the network.    Classification of sensors based on their importance in the network for structural observability under sensor failures is done in \cite{ComDioDoTri:09}. Papers \cite{WanGaoGaoDen:13}, \cite{RutRut:13} define indices for measuring the level of resilience of the network towards maintaining structural controllability. Robustness of a power grid towards maintaining structural controllability under $\gamma$ link failures is addressed in \cite{RamPeqAguKar:15}.  

Resilience of networks is addressed in  \cite{PasDorBul:13} by studying the various kinds of attacks, monitoring issues that can possibly lead to malfunctioning of the network, and attack detection mechanisms.  
Paper \cite{ComDio:07} addressed the optimal selection problem with the minimal placement of additional sensors  and among them those with minimal cost for structural observability. 
The complexity of the robust minimal controllability problem, where the goal is to determine a minimal subset of state variables to be actuated to ensure structural controllability
under additional constraints  is addressed in \cite{PeqRamSouPed:17}. Paper \cite{LiuMoPeqSinKarAgu:13} consider the minimum sensor placement problem when the sensors are subject to one sensor failure. Note that, the papers discussed above (\cite{RahAgh:13}-\cite{LiuMoPeqSinKarAgu:13}) address i/o selection for resilience towards maintaining structural controllability/observability and this paper focus on verification and design of resilient feedback matrix for arbitrary pole placement of the closed-loop poles.

Optimal cost feedback selection for LTI systems is addressed in literature for various instances (see \cite{PeqKarPap:15}, \cite{MooChaBel-018_line}, \cite{MooChaBel-018_erratum} and references therein). However,  papers \cite{PeqKarPap:15},  \cite{MooChaBel-018_line}, and \cite{MooChaBel-018_erratum} deal with optimal design and do not consider failure or malfunctioning of the links. The computational complexity of verifying that the closed-loop system has no structurally fixed modes (SFMs) is polynomial  when the feedback links are not subjected to failures  \cite{PicSezSil:84}. {\em In this paper, we show that verification of the no-SFM condition is NP-complete when the feedback links are subject to failure} (exhaustive search-based technique has  complexity exponential in the number of states of the system).  There has been some effort on the resilience of feedback  matrix.   Designing minimum cost resilient actuation-sensing-communication for
{\em regular descriptor systems}  while ensuring selective strong structural system's properties is addressed in \cite{PopPeqKarAguIli:17_arx}. The pairing of sensors and actuators to design a feedback pattern that is resilient to edge failures is assessed in  \cite{PeqKhoKriPap:17_arx}. The approach in  \cite{PeqKhoKriPap:17_arx} uses the notion of {\em resilient fixed modes} and gave conditions to verify non-existence of  resilient fixed modes when the subset of feedback links that can be compromised is specified. This paper, on the other hand, deals with the resilience of feedback matrix towards maintaining arbitrary pole placement when {\em any} subset of feedback links of cardinality at most $\gamma$ can fail. During an attack or failure, the subset of  feedback links  that can be compromised is arbitrary and may not be from a specified subset. Hence  algorithms to verify the resilience of a feedback matrix and algorithms to design resilient feedback matrix  that can handle the failure of any arbitrary subset of feedback links is important. This paper addresses these problems.

\section{ Problem Formulation}\label{sec:prob}
Consider an LTI dynamical system  $
\dot{x}(t)= Ax(t) + Bu(t)$,
$y(t)  = Cx(t)$, 
where the state matrix $A \in \R^{n \times n}$,  the input matrix  $B \in \R^{n \times m}$, and the output matrix $C \in \R^{p \times n}$. Here $\R$ denotes the set of real numbers. Consider  {\em structured matrices} $\bA \in \{\*, 0\}^{n \times n}$, $\bB \in \{\*, 0\}^{n \times m}$ and $\bC \in \{\*, 0\}^{p \times n}$. Here, $0$ denotes fixed zero entries and $\*$ denotes indeterminate free parameters. The tuple $(\bA, \bB, \bC)$ is said to be the {\it structured system} representation of the {\it numerical system} $(A$, $B$, $C)$ if it satisfies Eq.~\eqref{eq:struc} given below.
\begin{eqnarray}\label{eq:struc}
A_{ij} &=& 0 \mbox{~whenever~} \bA_{ij} = 0, \mbox{~and} \nonumber \\
B_{ij} &=& 0 \mbox{~whenever~} \bB_{ij} = 0, \mbox{~and} \nonumber \\
C_{ij} &=& 0 \mbox{~whenever~}\bC_{ij} = 0.
\end{eqnarray}
Here, $(\bA, \bB, \bC)$ represents a class of numerical systems that satisfy Eq.~\eqref{eq:struc}.  Let $\bK \in \{\*,0\}^{m \times p}$ denotes the structured feedback matrix, where $\bK_{ij} = \*$ if the $j^{\rm th}$ output is fed to the $i^{\rm th}$ input as feedback. For the structured feedback matrix $\bK$, the closed-loop structured system is denoted by $(\bA, \bB, \bC, \bK)$. The concept of fixed modes for structured systems is introduced and a necessary and sufficient graph-theoretic condition for checking the existence of {\em structurally fixed modes} (SFMs) is given in \cite{PicSezSil:84}. Let $[K]:= \{K:K_{ij}=0$, if $\bar{K}_{ij}=0\}$. Now we define structurally fixed modes.
\begin{definition}\cite{Rei:88}\label{def:sfm}
The structured system $(\bA, \bB, \bC)$ is said  to have no SFMs with respect to an information pattern $\bK$ if there exists  one numerical realization $(A, B, C)$ of $(\bA, \bB, \bC)$ such that $\cap_{K \in [K]} \sigma(A + BKC) = \emptyset$, where the function $\sigma(T)$ denotes the set of eigenvalues of a square matrix $T$.
\end{definition}
In this paper, we use the no-SFMs criteria to ensure arbitrary pole placement \cite{PicSezSil:84}, as no-SFMs criteria and the ability for arbitrary pole placement are equivalent when  controllers are dynamic \cite[Theorem~4.3.5]{shana_thesis}. We consider two problems in this paper, specifically in the context of the resilience of the closed-loop system towards achieving arbitrary pole placement property, which are described below.

\subsection{Feedback Resilience Verification Problem}
Consider a structured system $(\bA, \bB, \bC)$ and a structured feedback matrix $\bK$  such that the closed-loop system $(\bA, \bB, \bC, \bK)$ has no SFMs. Let  $\I \subset \{1,\ldots,m\} \times \{1,\ldots,p\}$ be a subset consisting of indices of entries of $\bK$. More precisely, $\I \subset \{(i,j): i\in\{1,\ldots,m\} \mbox{~and~} j \in \{1,\ldots, p\}\}$. Define $\bK^\I \in \{\*,0\}^{m \times p}$, where
\begin{equation}\label{eq:Ki}
{\bK^{\I}}_{ij} := \begin{cases} \bK_{ij}, \mbox{~if~} (i,j) \notin \I,\\
0,~~~ \mbox{~if~} (i,j) \in \I.
\end{cases}
\end{equation}
Now we formulate the first problem considered in this paper.
\begin{problem}[Feedback resilience verification problem]\label{prob:verification}
Consider a closed-loop structured system $(\bA, \bB, \bC, \bK)$ with no SFMs and a positive number $\gamma$. Then, verify if the structured system  $(\bA, \bB, \bC, \bK^{\I})$  guarantees no-SFMs criteria for \underline{any}  set $\I$, where $\I \subset \{(i,j): i\in\{1,\ldots,m\} \mbox{~and~} j \in \{1,\ldots, p\}\}$ such that $|\I| \leqslant \gamma$.
\end{problem}
For a given closed-loop system, the feedback resilience verification problem verifies if the system is resilient to failure of any set of feedback links with cardinality at most $\gamma$. Note that, the  set of feedback links that undergo failure or attack is not pre-specified and can be any arbitrary set. There are no known results to solve this problem.

\subsection{Sparsest Resilient Feedback Design Problem}
Now we describe the design problem associated with the resilience of the feedback matrix. The objective is to design a feedback matrix  $\bK$ such that for $\I \subset \{1,\ldots,m\} \times \{1,\ldots,p\}$, ${\bK^{\I}}$ defined in Eq.~\eqref{eq:Ki} satisfies no-SFMs criteria of $(\bA, \bB, \bC, {\bK^{\I}})$.

Let $\K_\gamma := \{\bK \in \{\*, 0 \}^{m \times p} :$ structured system  $(\bA, \bB, \bC, \bK^{\sI})$  guarantees no-SFMs criteria for {\em any}  set $\I$, where $\I \subset \{(i,j): i\in\{1,\ldots,m\} \mbox{~and~} j \in \{1,\ldots, p\}\}$ such that $|\I| \leqslant \gamma \}$. The set $\K_\gamma$ thus consists of all feedback matrices for the structured system $(\bA, \bB, \bC)$ that satisfy no-SFMs criteria even after the failure of any $\gamma$ feedback links. Without loss of generality, assume that $\bK^f=\{\bK^f_{ij}=\*$, for all $i,j\}$ lies in $\K_\gamma$, otherwise, no feasible solution to the problem. Thus $\K_\gamma$ is non-empty. Our objective here is to design a sparsest feedback matrix that  lies in $\K_\gamma$.

\begin{problem}[Sparsest resilient feedback  design problem]\label{prob:design}
Consider a structured system $(\bA, \bB, \bC)$ and a positive number $\gamma$. Then, find 
\[ \bK^\* ~\in~ \arg\min_{\bK \in \K_\gamma} \norm[\bK]_0, \]
\end{problem}
\noindent where $\K_\gamma := \{\bK \in \{0,\*\}^{m \times p} :$ structured system  $(\bA, \bB, \bC, \bK^{\sI})$  guarantees no-SFMs criteria for {\em any}  set $\I$, where $\I \subset \{(i,j): i\in\{1,\ldots,m\} \mbox{~and~} j \in \{1,\ldots, p\}\}$ such that $|\I| \leqslant \gamma \}$. Here $\norm[\cdot]_0$ denotes the zero matrix 
   norm\footnote{Although $\norm[\cdot]_0$ does not satisfy all 
   the norm axioms, the number of non-zero entries in a matrix is 
   conventionally referred to as the {\em zero norm}.}.

In the next section, we present few notations and existing results used in the sequel.

\section{Notations, Preliminaries and Existing Results}\label{sec:graph}
For understanding the graph-theoretic condition given in \cite{PicSezSil:84} that characterizes the no-SFMs criteria, we define few notations and constructions. Firstly, the {\em state  digraph} denoted by $\D(\bA) := (V_{\x}, E_{\x})$ is constructed as follows: here $V_{\x} = \{x_1, \ldots, x_n \}$ and $(x_j, x_i) \in E_{\x}$ if $\bA_{ij} = \*$. Subsequently, the {\em closed-loop system digraph} $\D(\bA, \bB, \bC, \bK) := (V_{\x}\cup V_{\u}\cup V_{Y}, E_{\x}\cup E_{\u}\cup E_{Y}\cup E_{K})$ is constructed, where $V_\u = \{u_1, \ldots, u_m\}$ and $V_Y = \{y_1, \ldots, y_p\}$. Further, $(u_j, x_i) \in E_{U}$ if $\bB_{ij} =\*$, $(x_j, y_i) \in E_{Y}$ if $\bC_{ij} = \*$ and $(y_j, u_i) \in E_K$ if $\bK_{ij} = \*$.  The digraph $\D(\bA, \bB, \bC, \bK)$ captures the effects of states, inputs, outputs and feedback connections in the system.

A digraph $\D = (V_D, E_D)$ is said to be {\em strongly connected} if for each ordered pair of vertices $v_i,v_j \in V_D$
there exists a directed path from $v_i$ to $v_j$. Further, a subgraph of digraph $\D$ denoted by $\D_S = (V_S, E_S)$ is a digraph such that $V_S \subset V_D$, $E_S \subset E_D$, and the edge set $E_S$ has endpoints from $V_S$ which is same as in $E_D$. A maximal strongly connected subgraph is a subgraph that is strongly connected and is not properly contained in any other subgraph that is strongly connected. 

\begin{definition}\cite{CorLeiRivSte:01} \label{def:SCC}
A strongly connected component (SCC) is a {\it maximal} strongly connected subgraph $\D_S = (V_S , E_S )$ of $\D$.
\end{definition}

Using $\D(\bA, \bB, \bC, \bK)$ and Definition~\ref{def:SCC} the following holds.

\begin{proposition} \cite[Theorem 4]{PicSezSil:84}\label{prop:SFM1} 
A structured system $(\bA, \bB, \bC)$ has no SFMs with respect to a feedback matrix $\bK$ if and only if the following conditions hold:\\
\noindent a)~in the digraph $\D(\bA, \bB, \bC, \bK)$, each state node $x_i$ is contained in an SCC which includes an edge in $E_K$, and \\
\noindent b)~there exists a finite node disjoint union of cycles $\C_g = (V_g, E_g)$ in $\D(\bA, \bB, \bC, \bK)$, where $g$ 
is a positive integer such that $V_X \subset \cup_{g}V_g$.
\end{proposition}

Condition~a) in Proposition~\ref{prop:SFM1} can be verified by finding all SCCs in $\D(\bA, \bB, \bC, \bK)$ and then checking if all of them have at least one feedback edge. Finding SCCs in a digraph $\D = (V_D, E_D)$ has complexity $O(|V_D|+|E_D|)$ \cite{Die:00}. Here, $m=O(n)$ and $p=O(n)$. Thus, the number of vertices and edges in $\D(\bA, \bB, \bC, \bK)$ are $O(n)$ and $O(n^2)$, respectively. As a result, condition~a) can be checked using $O(n^2)$ operations. There exists a graph-theoretic condition using the concept of information path for checking condition~b) in Proposition~\ref{prop:SFM1} in $O(n^{2.5})$ operations \cite{PapTsi:84}. In this paper, we use the bipartite graph matching condition given in \cite{MooChaBel:17_Automatica}. We now define  bipartite graphs   and then  give the bipartite matching condition to verify condition~b) in Proposition~\ref{prop:SFM1}.

A bipartite graph denoted by $G = (V, \widetilde{V}, \E)$, $|V| \leqslant |\tV|$, is a graph that satisfies $V \cap \widetilde{V} = \emptyset$ and $\E \subseteq V \times \widetilde{V}$. In $G$, a {\em matching} $M \subseteq \E$ is a collection of edges such that no two edges have a common endpoint and a {\em perfect matching} is a matching whose cardinality is $|V|$. Further, let $c: \E \rightarrow \R$ be a cost function. Then, a minimum cost perfect matching $M^\*$ is a perfect matching in $G$ such that $\sum_{e \in M^\*}c(e) \leqslant \sum_{e \in \widetilde{M}} c(e)$, where $\widetilde{M}$ is any perfect matching in $G$. Finding a minimum cost perfect matching in a bipartite graph has computational complexity $O(|V|^{2.5})$ \cite{Die:00}.

For a closed-loop structured system $(\bA,\bB,\bC,\bK)$, we construct a bipartite graph denoted by $\B(\bA,\bB,\bC,\bK)$.
Define  $\B(\bA,\bB,\bC,\bK):=(V_{X'}\cup V_{U'}\cup V_{Y'},V_{X}\cup V_{U}\cup V_{Y},\E_{X}\cup \E_{U}\cup \E_{Y} \cup \E_{K} \cup \E_{\mathbb{U}}\cup \E_{\mathbb{Y}})$, where $V_{X'}=\{x'_1,\dots,x'_n\}$, $V_{U'}=\{u'_1,\dots,u'_m\}$, $V_{Y'}=\{y'_1,\dots,y'_p\}$, $V_X=\{x_1,\dots,x_n\}$, $V_{U}=\{u_1,\dots,u_m\}$ and $V_{Y}=\{y_1,\dots,y_p\}$. Also, $(x'_j,x_i)\in \E_X \Leftrightarrow (x_i,x_j)\in {E}_X$, $(x'_i,u_j)\in \E_U \Leftrightarrow (u_j,x_i)\in E_U$, $(y'_j,x_i)\in \E_Y \Leftrightarrow (x_i,y_j)\in E_Y$ and $(u'_i,y_j)\in \E_K \Leftrightarrow (y_j,u_i)\in E_K$. Moreover, $\E_{\mathbb{U}}$ includes edges $(u'_i,u_i)$, for $i=1,\dots,m$ and $\E_{\mathbb{Y}}$ includes edges $(y'_i,y_i)$, for $i=1,\dots,p$. The bipartite graph  $\B(\bA,\bB,\bC,\bK)$ is referred as the {\em closed-loop system bipartite graph}. Using $\B(\bA,\bB,\bC,\bK)$, the following result holds.

\begin{proposition} \cite[Theorem~3]{MooChaBel:17_Automatica}, \label{prop:SFM2} 
Consider a closed-loop structured system $(\bA, \bB, \bC, \bK)$. Then, the bipartite graph $\B(\bA, \bB, \bC, \bK)$ has a perfect matching if and only if all state nodes are spanned by disjoint union of cycles in $\D(\bA, \bB, \bC, \bK)$.
\end{proposition}
An illustrative example demonstrating the construction of $\D(\bA, \bB, \bC, \bK)$ and $\B(\bA, \bB, \bC, \bK)$ for a structured system $(\bA, \bB, \bC, \bK)$  is given in Figure~\ref{fig:eg1}. Finding perfect matching has computational complexity $O(n^{2.5})$ \cite{Die:00}. Using Propositions~\ref{prop:SFM1} and~\ref{prop:SFM2}, one can verify if a structured system satisfies the no-SFMs criteria in $O(n^{2.5})$ computations. 

Note that, while Proposition~\ref{prop:SFM1} gives a polynomial-time graph-theoretic condition to verify the no-SFMs criteria, our objective is (i)~to verify if the structured system continues to satisfy  the no-SFMs condition even after the failure of {\em any} subset of feedback links with cardinality at most $\gamma$ and (ii)~to design a feedback matrix that guarantees the no-SFM criteria even after the failure of {\em any} subset of feedback links of cardinality at most $\gamma$.  In the next section, we  analyze the complexity of these two problems. 
\begin{figure}[t]
\centering
\begin{subfigure}[b]{0.45\textwidth}
\begin{equation*} \label{eq:sysmatrices}
\bA = 
\left[\scalebox{1.2}{\mbox{$
\begin{smallmatrix}
0 & 0 & \* & 0 & 0 & 0 & 0\\
0 & 0 & \* & 0 & 0 & 0 & 0\\
\* & \* & 0 & \* & \* &  0 & 0\\
0 & 0 & \* & 0 & 0 & 0 & 0\\
0 & 0 & \* & 0 & 0 & 0 & 0\\
0 & 0 & 0 & \* & 0 & 0 & \*\\
0 & 0 & 0 & 0 & 0 & \* &  0
\end{smallmatrix}
$}}
\right],~
\bB = 
\left[\scalebox{1.2}{\mbox{$
\begin{smallmatrix}
\* & 0  \\
0 & 0  \\
0 & 0  \\
0 & 0  \\
0 & 0  \\
0 & \*\\
0 & 0
\end{smallmatrix}
$}}
\right]
\end{equation*}
\begin{equation*}
\bC = 
\left[\scalebox{1.2}{\mbox{$
\begin{smallmatrix}
0 & 0 & 0 & 0 & \* & 0 & 0\\
0 & 0 & 0 & 0 & 0 & 0 & \*
\end{smallmatrix}
$}}
\right], ~
\linebreak
\bK = 
\left[\scalebox{1.2}{\mbox{$
\begin{smallmatrix}
\* & 0  \\
0 &  \*
\end{smallmatrix}
$}}
\right]
\end{equation*}
\hspace*{10 mm}
\begin{tikzpicture}[scale = 0.9, ->,>=stealth',shorten >=1pt,auto,node distance=1.85cm, main node/.style={circle,draw,font=\scriptsize\bfseries}]
\definecolor{myblue}{RGB}{80,80,160}
\definecolor{almond}{rgb}{0.94, 0.87, 0.8}
\definecolor{bubblegum}{rgb}{0.99, 0.76, 0.8}
\definecolor{columbiablue}{rgb}{0.61, 0.87, 1.0}

  \fill[bubblegum] (5,1.0) circle (6.0 pt);
  \fill[columbiablue] (5.0,-1) circle (6.0 pt);
  \fill[columbiablue] (8.25,1) circle (6.0 pt);

  \fill[almond] (5.0,0) circle (6.0 pt);
  \fill[almond] (5.75,0) circle (6.0 pt);
  \fill[almond] (6.5,0) circle (6.0 pt);
  \fill[almond] (5.75,1.0) circle (6.0 pt);
  \fill[almond] (5.75,-1.0) circle (6.0 pt);
  
  \fill[almond] (7.5,0) circle (6.0 pt);
  \fill[bubblegum] (7.5,1) circle (6.0 pt);
  \fill[almond] (8.25,0) circle (6.0 pt);

   \node at (5,1.0) {\scriptsize $u_1$};
   \node at (7.5,1.0) {\scriptsize $u_2$}; 
   
   \node at (5,-1) {\scriptsize $y_1$}; 
   \node at (8.25,1) {\scriptsize $y_2$};   

  \node at (5,0) {\scriptsize $x_2$};
  \node at (5.75,0) {\scriptsize $x_3$};
  \node at (6.5,0) {\scriptsize $x_4$};
  \node at (5.75,1.0) {\scriptsize $x_{1}$};
  \node at (5.75,-1.0) {\scriptsize $x_{5}$};
  
  \node at (7.5,0) {\scriptsize $x_6$};
  \node at (8.25,0) {\scriptsize $x_7$};

\draw (5.5,-1.0)  ->   (5.2,-1.0);
\draw (5.25,1.0)  ->   (5.55,1.0);
\draw (6.7,0)  ->   (7.3,0);
\draw (7.5,0.8)  ->   (7.5,0.2);
\draw (8.25,0.2)  ->   (8.25,0.8);
\draw [red] (8.05,1)  ->   (7.65,1);   

\path[every node/.style={font=\sffamily\small}]
(7.5,0.25) edge[bend left = 40] node [left] {} (8.25,0.25)
(8.25,-0.25) edge[bend left = 40] node [left] {} (7.5,-0.25)

(5,0.25) edge[bend left = 40] node [left] {} (5.75,0.25)
(5.75,-0.25) edge[bend left = 40] node [left] {} (5,-0.25)
(5.75,0.25) edge[bend left = 40] node [left] {} (6.5,0.25)
(6.5,-0.25) edge[bend left = 40] node [left] {} (5.75,-0.25)
(5.75,0.25) edge[bend left = 40] node [left] {} (5.55,1.05)
(5.95,1.05) edge[bend left = 40] node [left] {} (5.75,0.25)

(5,-0.8) edge[red, bend left = 40] node [left] {} (5,0.8)

(5.75,-0.25) edge[bend left = 40] node [left] {} (5.95,-1.05)
(5.55,-1.05) edge[bend left = 40] node [left] {} (5.75,-0.25);
\end{tikzpicture}
\caption{$\D(\bA, \bB, \bC, \bK)$}
\label{fig:digraph}
\end{subfigure}~\hspace{0.2 mm}
\begin{subfigure}[b]{0.45\textwidth}
\centering
\definecolor{myblue}{RGB}{80,80,160}
\definecolor{mygreen}{RGB}{80,160,80}
\definecolor{myred}{RGB}{144, 12, 63}
\definecolor{myyellow}{RGB}{214, 137, 16}
\definecolor{aquamarine}{rgb}{0.5, 1.0, 0.83}
\begin{tikzpicture} [scale = 0.18]               
          \node at (12,-5) {\scriptsize $x'_1$};
          \node at (12,-7.5) {\scriptsize $x'_2$};
          \node at (12,-10) {\scriptsize $x'_3$};
          \node at (12,-12.5) {\scriptsize $x'_4$};
          \node at (12,-15) {\scriptsize $x'_5$};
          \node at (12,-17.5) {\scriptsize $x'_6$};
          \node at (12,-20) {\scriptsize $x'_7$};
          \node at (12,-22.5) {\scriptsize $u'_1$};
          \node at (12,-25) {\scriptsize $u'_2$}; 
          \node at (12,-27.5) {\scriptsize $y'_1$};
          \node at (12,-30) {\scriptsize $y'_2$};            
         
          \fill[myblue] (13.5,-5) circle (15.0 pt); 
          \fill[myblue] (13.5,-7.5) circle (15.0 pt);
          \fill[myblue] (13.5,-10) circle (15.0 pt);
          \fill[myblue] (13.5,-12.5) circle (15.0 pt);
          \fill[myblue] (13.5,-15) circle (15.0 pt);
          \fill[myblue] (13.5,-17.5) circle (15.0 pt);
          \fill[myblue] (13.5,-20) circle (15.0 pt); 
          \fill[myred] (13.5,-22.5) circle (15.0 pt);
          \fill[myred] (13.5,-25) circle (15.0 pt);
          \fill[mygreen] (13.5,-27.5) circle (15.0 pt);
          \fill[mygreen] (13.5,-30) circle (15.0 pt);  
                    
\draw (14,-7.5)  --   (29.5,-10);
\draw (14,-5)  --   (29.5,-10);
\draw (14,-10)  --   (29.5,-7.5);
\draw (14,-10)  --   (29.5,-5);
\draw (14,-10)  --   (29.5,-12.5);
\draw (14,-10)  --   (29.5,-15);
\draw (14,-12.5)  --   (29.5,-10);
\draw (14,-15)  --   (29.5,-10);
\draw (14,-17.5)  --   (29.5,-12.5);
\draw (14,-17.5)  --   (29.5,-20);
\draw (14,-20)  --   (29.5,-17.5);

\draw [mygreen] (14,-27.5)  --   (29.5,-15);
\draw [mygreen] (14,-30)  --   (29.5,-20);

\draw [myred] (14,-5)  --   (29.5,-22.5);

\draw [mygreen] (14,-27.5)  --   (29.5,-27.5);
\draw [mygreen] (14,-30)  --   (29.5,-30);
\draw [myred] (14,-22.5)  --   (29.5,-22.5);
\draw [myred] (14,-25)  --   (29.5,-25);
%

\draw[myred] (14,-17.5)  --   (29.5,-25);
\draw[dashed] (14,-22.5)  --   (29.5,-27.5);
\draw[dashed] (14,-25)  --   (29.5,-30);

          \node at (31.5,-5) {\scriptsize $x_1$};
          \node at (31.5,-7.5) {\scriptsize $x_2$};
          \node at (31.5,-10) {\scriptsize $x_3$};
          \node at (31.5,-12.5) {\scriptsize $x_4$};
          \node at (31.5,-15) {\scriptsize $x_5$};
          \node at (31.5,-17.5) {\scriptsize $x_6$};
          \node at (31.5,-20) {\scriptsize $x_7$};
          \node at (31.5,-22.5) {\scriptsize $u_1$};
          \node at (31.5,-25) {\scriptsize $u_2$};
          \node at (31.5,-27.5) {\scriptsize $y_1$};
          \node at (31.5,-30) {\scriptsize $y_2$};
          
          \fill[myblue] (30,-5) circle (15.0 pt); 
          \fill[myblue] (30,-7.5) circle (15.0 pt);
          \fill[myblue] (30,-10) circle (15.0 pt);
          \fill[myblue] (30,-12.5) circle (15.0 pt);
          \fill[myblue] (30,-15) circle (15.0 pt);
          \fill[myblue] (30,-17.5) circle (15.0 pt);                                
          \fill[myblue] (30,-20) circle (15.0 pt);
          \fill[myred] (30,-22.5) circle (15.0 pt);
          \fill[myred] (30,-25) circle (15.0 pt);                                
          \fill[mygreen] (30,-27.5) circle (15.0 pt);                                
          \fill[mygreen] (30,-30) circle (15.0 pt);                                

        \node at (12,-3) {\scriptsize $V_{X'} \cup V_{U'} \cup V_{Y'}$};
        \node at (29,-3) {\scriptsize $V_{X} \cup V_{U} \cup V_{Y}$};
\end{tikzpicture}
\caption{$\B(\bA, \bB, \bC, \bK)$}
\label{fig:bigraph}
\end{subfigure}
\caption{The digraph $\D(\bA, \bB, \bC, \bK)$ and bipartite graph $\B(\bA, \bB, \bC, \bK)$ of the given closed-loop structured system $(\bA, \bB, \bC, \bK)$ is shown in Figure~\ref{fig:digraph} and Figure~\ref{fig:bigraph}, respectively.}
\label{fig:eg1}
\end{figure}
\section{Complexity Results}\label{sec:NP}
In this section, we analyze the complexity of Problem~\ref{prob:verification} and Problem~\ref{prob:design}. We prove Problem~\ref{prob:verification} is NP-complete  using a known NP-complete problem, the {\em blocker problem}. We  prove that   Problem~\ref{prob:design} is  NP-hard using {\em minimum set multi-covering problem}, a known NP-hard problem. First, we prove NP-completeness of Problem~\ref{prob:verification}.

\subsection{Complexity of Feedback Resilient Verification Problem}\label{subsec:verif}
The NP-completeness result for the feedback resilient verification problem (Problem~\ref{prob:verification}) is obtained by reducing a known NP-complete problem, the {\em blocker problem}, to an instance of Problem~\ref{prob:verification}.  Now we describe the blocker problem. 

\begin{problem}[Blocker problem: $Block(G,1,\gamma)$]\label{prob:two}
Given a bipartite graph $G := (V, \tV, \E)$ with $|V| \leqslant |\tV|$, does there exist a set $T \subseteq \E$ with $|T| \leqslant \gamma$ such that $\upsilon(G') \leqslant \upsilon(G)-1 $, where $\upsilon(G)$ and $\upsilon(G')$ denote the size of the maximum matching in $G$ and $G'$, respectively with $G':=(V, \tV ,\E\setminus T)$.
\end{problem}
 
$Block(G,1,\gamma)$ is NP-complete \cite{ZenReiPic:09}. Note that $Block(G,1,\gamma)$ is NP-complete even when $G$ has  perfect matching \cite{ZenReiPic:09}. Thus, we have the following proposition.

\begin{proposition}\cite[Theorem~3.3]{ZenReiPic:09}\label{prop:blocker}
Consider a bipartite graph $G=(V, \tV, \E)$ with $|V| \leqslant |\tV|$. Let $\upsilon(G) = |V|$, i.e., $G$ has a perfect matching. Then, $Block(G,1,\gamma)$ is NP-complete. 
\end{proposition}
Now, we reduce a general instance of blocker problem to an instance of Problem~\ref{prob:verification} and then prove that Problem~\ref{prob:verification} is NP-complete.
\begin{algorithm}[h]
\caption{Pseudocode showing reduction of the blocker problem to an instance of Problem~\ref{prob:verification} \label{alg:reduction}}
  \begin{algorithmic}
\State \textit {\bf Input:} General bipartite graph $G=(V,\tV,\E)$ with $|V|=r$ and $|\tV|=s$ 
\State \textit{\bf Output:} $\bA \in \{0, \*\}^{(s+2) \times (s+2)}$, $\bB \in \{0, \*\}^{(s+2) \times s}$, $\bC \in \{0, \*\}^{r\times (s+2)}$ and $\bK \in \{0, \*\}^{s \times r}$ 
\end{algorithmic}
  \begin{algorithmic}[1]
  \State Define $(\bA, \bB, \bC, \bK)$ connected as follows:
  \State  $\bA_{ij} \leftarrow \begin{cases}
    \*$, for $i=2$ and $j \in \{1,\ldots,s+2\}, \\
    \*$, for $i \in \{1,\ldots,s+2\}$ and $j=2, \label{step:A}\\
    \*$, for $i \in \{3,\ldots,s-r+2\}$ and $j \in \{3,\ldots,s+2\},\\
    0, \mbox{~otherwise}.
  \end{cases} $
   \State  $\bB_{ij} \leftarrow \begin{cases}
    \*$, for $i \in \{s-r+3,\ldots,s+2\} $ and $ j \in \{1,\ldots,s\} , \label{step:B}\\
    0, \mbox{~otherwise}.
  \end{cases} $ 
  \State  $\bC_{ij} \leftarrow \begin{cases}
    \*$, for $i \in \{1,\ldots,r\} $ and $j \in \{3,\ldots,s+2\} , \label{step:C}\\
    0, \mbox{~otherwise}.
  \end{cases} $
  \State  $\bK_{ij} \leftarrow \begin{cases}
    \*$, for all $(v_{j},\tv_{i}) \in \E , \label{step:K}\\
    0, \mbox{~otherwise}.
  \end{cases} $
\end{algorithmic}
\end{algorithm}

Algorithm~\ref{alg:reduction} gives a reduction of $Block(G,1,\gamma)$  to an instance of Problem~\ref{prob:verification}. Given a general instance of the blocker problem, i.e., $G=(V,\tV,\E)$, where $V=\{v_{1},\ldots,v_{r}\}$, $\tV=\{\tv_{1},\ldots,\tv_{s}\}$, $s \geqslant r$ and $1\leqslant \gamma \leqslant |\E|$, we construct a closed-loop  structured system $(\bA, \bB, \bC, \bK)$  with $(s+2)$ number of states, $s$ number of inputs and $r$ number of outputs. In Step~\ref{step:A}, we define the set $E_{X}$ as follows: there exist directed edges from node $x_{2}$ to every node in the set $\{x_{1},\ldots,x_{s+2}\}$, from every node in the set $\{x_{1},\ldots,x_{s+2}\}$ to node $x_{2}$ and from every node in $\{x_{3},\ldots,x_{s+2}\}$ to every node in $\{x_{3},\ldots,x_{s-r+2}\}$. By this construction of $\bA$, $\D(\bA)$ is an irreducible graph (see Figure~\ref{fig:Digraph_schematic}). In Step~\ref{step:B}, we construct edge set $E_{U}$ as follows: a directed edge exists  from every input in $\{u_{1},\ldots,u_{s}\}$ to every state node in $\{x_{s-r+3},\ldots,x_{s+2}\}$. Thus in the $\bB$ constructed, no input directly actuates states $\{x_1,\ldots, x_{s-r+2}\}$. In Step~\ref{step:C}, the output edge set $E_{Y}$ is constructed in such a way that every state node in $\{x_{3},\ldots,x_{s+2}\}$ is connected to every output node in $\{y_{1},\ldots,y_{r}\}$. Thus in the $\bC$ constructed, states $\{x_1, x_2\}$ cannot be sensed directly. In Step~\ref{step:K}, the feedback edges are constructed in such a way that for every edge $(v_{i},\tv_{j})$ $\in \E$ in $G$ there is an edge $(y_{i},u_{j})$ in $E_K$. This completes the construction of $\bK$.
An illustrative example demonstrating the  construction of the structured system $(\bA, \bB, \bC, \bK)$ for a given instance of the blocker problem is given in Figure~\ref{fig:eg}.

\begin{figure}[t]
\begin{subfigure}[b]{0.475\textwidth}
\centering
\definecolor{myblue}{RGB}{80,80,160}
\definecolor{mygreen}{RGB}{80,160,80}
\definecolor{myred}{RGB}{144, 12, 63}
\definecolor{myyellow}{RGB}{214, 137, 16}
\definecolor{aquamarine}{rgb}{0.5, 1.0, 0.83}
\begin{tikzpicture} [scale = 0.18]               
          \node at (10,-5) {\scriptsize $v_1$};
          \node at (10,-7.5) {\scriptsize $v_2$};           
         
          \fill[myblue] (12,-5) circle (15.0 pt); 
          \fill[myblue] (12,-7.5) circle (15.0 pt); 
                    
\draw (12.5,-5)  --   (31.5,-5);
\draw (12.5,-5)  --   (31.5,-10);
\draw (12.5,-7.5)  --   (31.5,-7.5);
\draw (12.5,-7.5)  --   (31.5,-5);

 \node at (34,-5) {\scriptsize $\tv_1$};
 \node at (34,-7.5) {\scriptsize $\tv_2$};
 \node at (34,-10) {\scriptsize $\tv_3$};
          
 \fill[myblue] (32,-5) circle (15.0 pt); 
 \fill[myblue] (32,-7.5) circle (15.0 pt);
 \fill[myblue] (32,-10) circle (15.0 pt);

 \node at (12,-3) {\scriptsize $V$};
 \node at (32,-3) {\scriptsize $\tV$};
\end{tikzpicture}
\caption{An example of bipartite graph $G= (V, \tV, \E)$}
\label{fig:illus}
\end{subfigure}
\vspace{5 mm}
\begin{subfigure}[b]{0.475\textwidth}
\centering
\definecolor{myblue}{RGB}{80,80,160}
\definecolor{mygreen}{RGB}{80,160,80}
\definecolor{myred}{RGB}{144, 12, 63}
\definecolor{myyellow}{RGB}{214, 137, 16}
\definecolor{aquamarine}{rgb}{0.5, 1.0, 0.83}
\begin{tikzpicture} [scale = 0.18]               
          \node at (10,-5) {\scriptsize $x'_1$};
          \node at (10,-7.5) {\scriptsize $x'_2$};
          \node at (10,-10) {\scriptsize $x'_3$};
          \node at (10,-12.5) {\scriptsize $x'_4$};
          \node at (10,-15) {\scriptsize $x'_5$};
          \node at (10,-17.5) {\scriptsize $u'_1$};
          \node at (10,-20) {\scriptsize $u'_2$};
          \node at (10,-22.5) {\scriptsize $u'_3$};
          \node at (10,-25) {\scriptsize $y'_1$}; 
          \node at (10,-27.5) {\scriptsize $y'_2$};           
         
          \fill[myblue] (12,-5) circle (15.0 pt); 
          \fill[myblue] (12,-7.5) circle (15.0 pt);
          \fill[myblue] (12,-10) circle (15.0 pt);
          \fill[myblue] (12,-12.5) circle (15.0 pt);
          \fill[myblue] (12,-15) circle (15.0 pt);
          \fill[myred] (12,-17.5) circle (15.0 pt);
          \fill[myred] (12,-20) circle (15.0 pt); 
          \fill[myred] (12,-22.5) circle (15.0 pt);
          \fill[mygreen] (12,-25) circle (15.0 pt);
          \fill[mygreen] (12,-27.5) circle (15.0 pt);  
                    
\draw (12.5,-7.5)  --   (31.5,-5);
\draw (12.5,-5)  --   (31.5,-7.5);
\draw (12.5,-7.5)  --   (31.5,-7.5);
\draw (12.5,-7.5)  --   (31.5,-10);
\draw (12.5,-7.5)  --   (31.5,-12.5);
\draw (12.5,-7.5)  --   (31.5,-15);
\draw (12.5,-10)  --   (31.5,-7.5);
\draw (12.5,-10)  --   (31.5,-10);
\draw (12.5,-10)  --   (31.5,-12.5);
\draw (12.5,-10)  --   (31.5,-15);
\draw (12.5,-12.5)  --   (31.5,-7.5);
\draw (12.5,-15)  --   (31.5,-7.5);

\draw [myblue] (12.5,-12.5)  --   (31.5,-17.5);
\draw [myblue] (12.5,-15)  --   (31.5,-17.5);
\draw [myblue] (12.5,-12.5)  --   (31.5,-20);
\draw [myblue] (12.5,-15)  --   (31.5,-20);
\draw [myblue] (12.5,-12.5)  --   (31.5,-22.5);
\draw [myblue] (12.5,-15)  --   (31.5,-22.5);

\draw [mygreen] (12.5,-25)  --   (31.5,-10);
\draw [mygreen] (12.5,-25)  --   (31.5,-12.5);
\draw [mygreen] (12.5,-25)  --   (31.5,-15);
\draw [mygreen] (12.5,-27.5)  --   (31.5,-10);
\draw [mygreen] (12.5,-27.5)  --   (31.5,-12.5);
\draw [mygreen] (12.5,-27.5)  --   (31.5,-15);

\draw [myred] (12.5,-17.5)  --   (31.5,-17.5);
\draw [myred] (12.5,-20)  --   (31.5,-20);
\draw [myred] (12.5,-22.5)  --   (31.5,-22.5);

\draw [mygreen] (12.5,-27.5)  --   (31.5,-27.5);
\draw [mygreen] (12.5,-25)  --   (31.5,-25);
%

\draw[dashed] (12.5,-17.5)  --   (31.5,-25);
\draw[dashed] (12.5,-17.5)  --   (31.5,-27.5);
\draw[dashed] (12.5,-20)  --   (31.5,-27.5);
\draw[dashed] (12.5,-22.5)  --   (31.5,-25);

          \node at (34,-5) {\scriptsize $x_1$};
          \node at (34,-7.5) {\scriptsize $x_2$};
          \node at (34,-10) {\scriptsize $x_3$};
          \node at (34,-12.5) {\scriptsize $x_4$};
          \node at (34,-15) {\scriptsize $x_5$};
          \node at (34,-17.5) {\scriptsize $u_1$};
          \node at (34,-20) {\scriptsize $u_2$};
          \node at (34,-22.5) {\scriptsize $u_3$};
          \node at (34,-25) {\scriptsize $y_1$};
          \node at (34,-27.5) {\scriptsize $y_2$};
          
          \fill[myblue] (32,-5) circle (15.0 pt); 
          \fill[myblue] (32,-7.5) circle (15.0 pt);
          \fill[myblue] (32,-10) circle (15.0 pt);
          \fill[myblue] (32,-12.5) circle (15.0 pt);
          \fill[myblue] (32,-15) circle (15.0 pt);
          \fill[myred] (32,-17.5) circle (15.0 pt);                                
          \fill[myred] (32,-20) circle (15.0 pt);
          \fill[myred] (32,-22.5) circle (15.0 pt);
          \fill[mygreen] (32,-25) circle (15.0 pt);                                
          \fill[mygreen] (32,-27.5) circle (15.0 pt);                                

        \node at (12,-3) {\scriptsize $V_{X'} \cup V_{U'} \cup V_{Y'}$};
        \node at (32,-3) {\scriptsize $V_{X} \cup V_{U} \cup V_{Y}$};
\end{tikzpicture}
\caption{The bipartite graph representation $\B(\bA, \bB, \bC, \bK)$ for the blocker problem given in Figure~\ref{fig:illus}}
\label{fig:system_bigraph}
\end{subfigure}
\caption{Illustrative example demonstrating the reduction given in Algorithm~\ref{alg:reduction}.}
\label{fig:eg}
\end{figure}
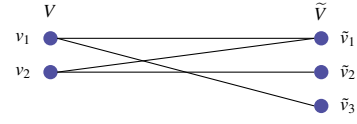
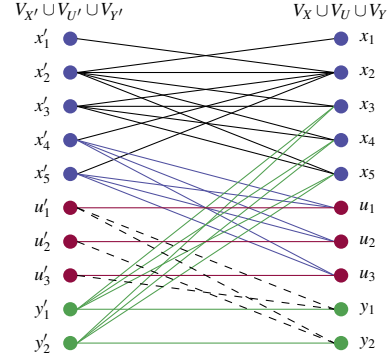
Next, we prove the following result.
\begin{lemma}\label{lem:PM}
Consider a bipartite graph $G = (V, \tV, \E)$ and  let $(\bA, \bB, \bC, \bK)$ be the closed-loop structured system constructed using Algorithm~\ref{alg:reduction}. Then, if there exists a perfect matching in $G$, then there exists a perfect matching in $\B(\bA, \bB, \bC, \bK)$. 
\end{lemma}

\begin{proof}
Let $M_{G}$ be a perfect matching in $G$. Since $r \leqslant s$, $|M_{G}|=r$. To prove the result, we extend the matching $M_{G}$ to a perfect matching in the bipartite graph $\B(\bA, \bB, \bC, \bK)$ as follows: from the construction of the structured system $(\bA, \bB,\bC,\bK)$, corresponding to every edge $(v_{i},\tv_{j}) \in \E$ there exists an edge $(u'_{j},y_{i}) \in \E_K$ in $\B(\bA, \bB, \bC, \bK)$. Hence there exists a matching of size $r$, say $M_{1} \subseteq \E_K$, in $\B(\bA, \bB, \bC, \bK)$.	 Notice that in $M_{1}$, $r$ vertices in $V_{Y}$ and $r$ vertices in $V_{U'}$ are matched. With respect to $M_{1}$, $(s-r)$ vertices in $V_{U'}$ of $\B(\bA, \bB, \bC, \bK)$ are unmatched. Let $M_{2} \subseteq \E_{\mathbb{U}}$ be a matching in $\B(\bA, \bB, \bC, \bK)$ that consists of edges that match those $(s-r)$ unmatched vertices in $\B(\bA, \bB, \bC, \bK)$. Hence $M_{1}\cup M_{2}$ is a matching in $\B(\bA, \bB, \bC, \bK)$ of size $s$. Corresponding to $M_{1}\cup M_{2}$, $r$ vertices in $V_{U}$ are unmatched. The only possible way to extend $M_{1}\cup M_{2}$ to a matching where these $r$ vertices are matched is by connecting them to vertices in $V_{X'}$ (since every vertex in $V_{U'}$ is matched in $M_1 \cup M_2$). In the construction of the bipartite graph $\B(\bA, \bB, \bC, \bK)$, all vertices in $V_{U}$ have edges to all the vertices in $\{x'_{s-r+3},\ldots, x'_{s+2}\}$. Notice that $|\{x'_{s-r+3},\ldots, x'_{s+2}\}|=r$. Hence we construct a	 matching $M_{3}$ between $r$ vertices in $V_{U}$ and $r$ vertices in $\{x'_{s-r+3},\ldots, x'_{s+2}\}$. Thus $M_{1}\cup M_{2}\cup M_{3}$ is a matching of size $(s+r)$ in $\B(\bA, \bB, \bC, \bK)$. With respect to the matching $M_{1}\cup M_{2}\cup M_{3}$, $(s-r)$ vertices  $\{x'_3, \ldots, x'_{s-r+2}\}$, two vertices $\{x'_{1},x'_{2}\}$ and $r$ vertices $\{y'_1, \ldots, y'_r\}$ of $\B(\bA, \bB, \bC, \bK)$ are unmatched on the left side. Thus the total unmatched vertices in the left side add up to $(s+2)$. On the right side, $(s+2)$ vertices $\{x_{1}, \ldots, x_{s+2}\}$ are unmatched. Notice that, all  vertices in $\{x'_{3}, \ldots, x'_{s-r+2}\}$ and all vertices in $\{y'_1, \ldots, y'_r\}$ are connected to all vertices in $\{x_{3}, \ldots, x_{s+2}\}$. Define a matching $M_{4}$ between vertices $\{x'_{3}, \ldots, x'_{s-r+2},  y'_1, \ldots, y'_r\}$ and $\{x_{3}, \ldots, x_{s+2}\}$. Thus $|M_4| = s$ and $|M_{1}\cup M_{2}\cup M_{3}\cup M_{4}| = 2s+r$. Further, $\{(x'_{1},x_{2}), (x'_{2},x_{1})\} \in \E_X$.  Thus $M_{1}\cup M_{2}\cup M_{3}\cup M_{4}\cup \{(x'_{1},x_{2}),(x'_{2},x_{1})\}$ is a perfect matching in $\B(\bA, \bB, \bC, \bK)$. Thus, if there exists a perfect matching in $G$, then there exists a perfect matching in $\B(\bA, \bB, \bC, \bK)$ for the structured system $(\bA, \bB, \bC, \bK)$ constructed in Algorithm~\ref{alg:reduction}.
\end{proof}

\noindent The  result below gives the complexity of Problem~\ref{prob:verification}.
\begin{theorem}\label{th:NP}
Consider a bipartite graph $G=(V,\tV, \E)$ with a perfect matching. Let $(\bA, \bB, \bC, \bK)$ be the closed-loop structured system constructed using Algorithm~\ref{alg:reduction}. Then, there exists a blocker $T$ of size $\gamma$ in $G$  if and only if there exists a blocker $\widetilde{T} \subseteq \E_K$ of size $\gamma$ in bipartite graph $\B(\bA, \bB, \bC, \bK)$. Moreover, Problem~\ref{prob:verification} is NP-complete.
\end{theorem}
 
\begin{proof} 
{\bf If part:} Here, we assume $\B(\bA, \bB, \bC, \bK)$ has a blocker $\widetilde{T} \subseteq \E_K$, $|\widetilde{T}| = \gamma$ and then prove that $G$ has a blocker $T$ such that $|T| = \gamma$. To the contrary, assume that there exists no blocker in $G$ of size $\gamma$. Thus, there exists a perfect matching in $G$ even after removing any of the $\gamma$ edges in $\E$. From Lemma~\ref{lem:PM}, if there exists a perfect matching in $G$, then there exists a perfect matching in $\B(\bA, \bB, \bC, \bK)$. But this contradicts the assumption that $\B(\bA, \bB, \bC, \bK)$ has a blocker. This completes the if-part.

\noindent {\bf Only if part:} Here, we assume there exists a blocker in $G$ of size $\gamma$ and prove that there exists a blocker $\widetilde{T} \subseteq \E_K$ with $|\widetilde{T}| = \gamma$ in $\B(\bA, \bB, \bC, \bK)$. To the contrary, assume that there exists no blocker $\widetilde{T} \subseteq \E_K$ in $\B(\bA, \bB, \bC, \bK)$. Then there exists a perfect matching, say $M_{\B}$, in $\B(\bA, \bB, \bC, \bK)$ even after removing all the edges in $\widetilde{T}$. Notice that $\{(x'_1, x_2), (x'_2, x_1)\} \in M_\B$, since vertices $x'_1$ and $x_1$ are connected only to vertices $x_2$ and $x'_2$ respectively in $\B(\bA, \bB,\bC,\bK)$. Consider the vertices $\{x'_3, \ldots, x'_{s-r+2}\}$.  Note that, these vertices can be only matched to vertices in $V_X \setminus \{x_1, x_2\}$, since there are no vertices in $\B(\bA, \bB, \bC, \bK)$ connecting them as $x_2$ is already matched to $x'_1$. After we match vertices $\{x'_3, \ldots, x'_{s-r+2}\}$ to $(s-r)$ vertices in $\{x_3, \ldots, x_{s+2}\}$, we are left with $r$ vertices which are unmatched in $\{x_3, \ldots, x_{s+2}\}$. The only possible way in which these $r$ vertices can be matched is by using edges from $\E_Y$ as other vertices which have edges with them, i.e., $\{x'_2, \ldots, x'_{s-r+2}\}$, are already matched. Since $|V_{Y'}|=r$, all the vertices $\{y'_1,\ldots,y'_r\}$ must be matched to $r$ remaining vertices in $V_X$. Now consider the vertices $\{y_1,\ldots,y_r\}$. As $\{y'_1,\ldots,y'_r\}$ are already matched we are left with the only option of matching these vertices to vertices in $V_{U'}$. Hence, there should be a matching of size $r$ to match all the vertices $\{y_1,\ldots,y_r\}$ using edges only from $\E_K$. So in $M_\B$ there exists $M'_\B \subset M_\B$ and $M'_\B \subseteq \E_K$ such that $|M'_\B| = r$. From the construction of the structured system $(\bA, \bB, \bC, \bK)$, notice that $(u'_{j}, y_{i}) \in \E_K$ if $(v_{i},\tv_{j})\in \E$. Thus, there exists a matching of size $|M'_\B|$ in $G$. Further, this is a perfect matching in $G$ as $|M'_\B|=r$. But this contradicts the assumption that $G$ has a blocker. Hence there is a blocker in $\B(\bA, \bB, \bC, \bK)$. 

Using the if-part and the only-if part and from Proposition~\ref{prop:blocker},  checking if the bipartite graph $\B(\bA, \bB, \bC, \bK)$ has a perfect matching after removing any set of feedback links of size at most $\gamma$   is NP-complete. Hence, Problem~\ref{prob:verification} is NP-complete.
\end{proof}

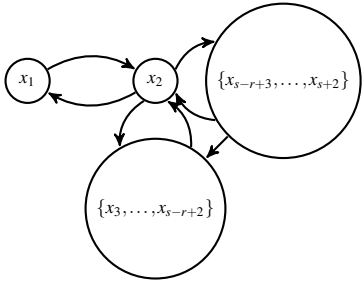
\begin{figure}[t]
\centering
\begin{tikzpicture}[scale=0.07, ->,>=stealth',shorten >=0.3pt,auto,node distance=1.7cm,
                thick,main node/.style={circle,draw,font=\scriptsize\bfseries}]

  \node[main node] (1) {$x_1$};
  \node[main node] (2) [right of=1] {$x_2$};
  \node[main node] (3) [right of=2] {$\{x_{s-r+3}, \ldots, x_{s+2}\}$};
  \node[main node] (4) [below of=2] {$\{x_{3}, \ldots, x_{s-r+2}\}$};
\draw[] (3) -> (4); 
\path[every node/.style={font=\sffamily\small}]
(1) edge[bend left = 30] node [left] {} (2) 
(2) edge[bend left = 30] node [left] {} (1)
(2) edge[bend left = 30] node [left] {} (3) 
(3) edge[bend left = 30] node [left] {} (2)
(2) edge[bend right = 30] node [left] {} (4)
(4) edge[bend right = 30] node [left] {} (2);  
 \end{tikzpicture}
\caption {Schematic diagram showing construction of digraph $\D(\bA)$ for the state matrix $\bA$ constructed  in Algorithm~\ref{alg:reduction}. All vertices in the above figure is strongly connected. Thus $\D(\bA)$ is irreducible.} 
\label{fig:Digraph_schematic}
\end{figure}

For the state matrix $\bA$ constructed in Algorithm~\ref{alg:reduction}, $\D(\bA)$ is irreducible:  for any $i,j \in \{1,\ldots, s+2\}$, there exists a path from $x_{i}$ to $x_{j}$ through node $x_{2}$. 
A schematic diagram that shows the digraph $\D(\bA)$ for the $\bA$ given in Algorithm~\ref{alg:reduction} is given in Figure~\ref{fig:Digraph_schematic}. The following result is an immediate consequence of Theorem~\ref{th:NP}.

\begin{cor}\label{cor:NP_irreducible}
Consider a closed-loop structured system $(\bA, \bB, \bC, \bK)$. Then, Problem~\ref{prob:verification} is NP-complete  when $\D(\bA)$ is irreducible.
\end{cor}

Now let us consider Problem~\ref{prob:verification} for  systems in which all state nodes are spanned by disjoint union of cycles. In other words, condition~b) in Proposition~\ref{prop:SFM1} is satisfied without using any feedback connections.  This class of systems is called  {\em structurally cyclic} systems \cite{MooChaBel-018_erratum}.  There is a wide class of systems so-called  {\em self-damped} systems that include multi-agent systems and epidemic systems, that are structurally cyclic  \cite{ChaMes:13}. 

\begin{lemma}\label{lem:scc}
Consider a closed-loop structured system $(\bA, \bB, \bC, \bK)$ in which all state nodes are spanned by disjoint union of cycles. Then, Problem~\ref{prob:verification} is solvable in $O(n^2)$ operations.
\end{lemma}
\begin{proof}
In structurally cyclic  systems (systems in which all state nodes are spanned by disjoint union of cycles), condition~b) is satisfied without using any feedback edges. In order to  maintain the no-SFMs criteria, the closed-loop system must maintain condition~a) of Proposition~\ref{prop:SFM1} even after removing any $\gamma$ feedback links. This can be  verified in $O(n^2)$ operations  by finding all  SSCs of $\D(\bA, \bB, \bC, \bK)$ and checking whether each SSC has at least $(\gamma+1)$ feedback links. Finding SCCs in a digraph has $O(n^2)$ operations \cite{Die:00} and hence the proof follows.
\end{proof}

Subsection~\ref{subsec:verif} concludes that Problem~\ref{prob:verification} is NP-complete for general structured  systems and structured systems whose state digraph $\D(\bA)$ is irreducible.  However,  Problem~\ref{prob:verification} is polynomial-time solvable when condition~b) in Proposition~\ref{prop:SFM1} is satisfied (Lemma~\ref{lem:scc}). In the next subsection, we analyze the complexity of Problem~\ref{prob:design}.

\subsection{Complexity of Sparsest Resilient Feedback  Design Problem}

In this section, we analyze the  complexity of Problem~\ref{prob:design}. Firstly, we claim that Problem~\ref{prob:design} is NP-hard for general systems. This result is a consequence of Theorem~\ref{th:NP} as Problem~\ref{prob:verification} which is the decision problem \cite{GarJoh:02} corresponding to the optimization problem, Problem~\ref{prob:design}, is NP-complete.  
\begin{cor}\label{cor:design_NP}
Consider a structured system $(\bA, \bB, \bC)$. Then, Problem~\ref{prob:design} is NP-hard.
\end{cor}
From Corollary~\ref{cor:NP_irreducible} we also infer that Problem~\ref{prob:design} is NP-hard even when $\D(\bA)$ is irreducible. However, the complexity of Problem~\ref{prob:design}  is not straightforward for structurally cyclic systems, i.e., the class of systems whose state nodes are spanned by disjoint cycles.  In this section, we show that Problem~\ref{prob:design} is NP-hard for  structurally cyclic systems, while Problem~\ref{prob:verification}  is polynomial-time solvable for structurally cyclic systems (Lemma~\ref{lem:scc}).

The NP-hardness result of Problem~\ref{prob:design} for structurally cyclic systems is obtained using reduction from {\em minimum set multi-covering problem} (MSMC). The MSMC problem  is described in Problem~\ref{prob:MSMC} for the sake of completeness. We first define a {\em cover} in Definition~\ref{def:cover}.

\begin{definition}\label{def:cover}
 Given a set of $N$ elements $\U=\{1,\ldots,N\}$ referred to as {\em universe} and collection of $r$ {\em sets} $\P=\{\S_1,\ldots,\S_r\}$, with $\S_i \subseteq \U$, for all $i \in \{1,\ldots,r\}$, such that $\cup_{i=1}^r \S_i =\U$, a {\em cover} $\hat{\S} \subseteq \P$ satisfies $\cup_{\S_i \in \hat{\S}} \S_i = \U$.
 \end{definition}
 
\begin{problem}[MSMC problem $(\U,\P,\alpha)$]\label{prob:MSMC}
Given $\U=\{1,\ldots,N\}$, $\P=\{\S_1,\ldots,\S_r\}$ satisfying $\S_i \subseteq \U$, for all $i \in \{1,\ldots,r\}$, and $\cup_{i=1}^r \S_i =\U$, and a constant demand $\alpha$. The MSMC problem consists of finding a set of indices $\I^\* \subseteq \{1,\ldots,r\}$ corresponding to the minimum number of sets covering $\U$, where every element $i \in \U$ is covered at least $\alpha$ times, i.e., 
\[ \J^\* = \arg\min_{\J \subseteq \{1,\ldots,r\}} |\J|\] 
such that $\cup_{j \in \J} \S_j = \U ~and~ ~|\{j \in \J : i \in \S_j\}| \geqslant \alpha$.
\end{problem} 
A cover $\hat{\S}$ that is a feasible solution to Problem~\ref{prob:MSMC} is called as a multi-cover. To prove the NP-hardness of Problem~\ref{prob:design}, we now present a reduction of a general instance of the MSMC problem to an instance of Problem~\ref{prob:design}.
\begin{algorithm}[t]
  \caption{Psuedo-code for reducing the MSMC problem to an instance of Problem~\ref{prob:design} 
  \label{alg:reduc_from_multisetcover}}
  \begin{algorithmic}
  	\State \textit {\bf Input:} MSMC problem with universe $\U=\{1,\ldots,N\}$, sets $\P=\{S_1,\ldots,S_r\}$ and constant demand $\alpha$
	\State \textit{\bf Output:} Structured system $(\bA, \bB, \bC)$ and a positive number $\gamma$
  \end{algorithmic}
  \begin{algorithmic}[1]
  	\State Define $x_1,\ldots,x_N$, $y_1,\ldots,y_p$ and $u_1$ to be interconnected by the following definition of $(\bA, \bB, \bC)$:\label{step:state_inp_oup}
  	\State  $\bA_{ij} \leftarrow \begin{cases}
    \*$, for $i=j$ and $i, j \in \{1,\ldots,N\}, \\
    0, \mbox{~otherwise}.
  \end{cases} $ \label{step:state_conn}
  \State $\bB_{i1} \leftarrow \begin{cases}
    \*$, for $i \in \{1,\ldots,N\}, \\
    0, \mbox{~otherwise}.
  \end{cases} $ \label{step:inp_conn}
  \State $\bC_{ij} \leftarrow \begin{cases}
    \*$, for $j \in \S_i$ and $i \in \{1, \ldots, r\}, \\
    0, \mbox{~otherwise}.
  \end{cases} $ \label{step:oup_conn} 
  \State Given a solution $\bK$ to Problem~\ref{prob:design} on $(\bA,\bB,\bC)$ and $\gamma = \alpha -1$, define sets selected under $\bK$, $\S(\bK) \leftarrow \{\S_j:\bK_{1j} = \*\}$  \label{step:sol_to_msp}
  \end{algorithmic}
\end{algorithm}

\begin{figure}[t]
\centering
\begin{subfigure}[b]{0.45\textwidth}
\begin{eqnarray*} \label{eq:sys2}
\U &=& \{1,2,3,4,5\}\\
\S_1 & = & \{1,2\},~\S_2=\{2,3\},~\S_3=\{3,4,5\}\\
\bA &=& 
\left[\scalebox{1.05}{\mbox{$
\begin{smallmatrix}
\* & 0 & 0 & 0 & 0 \\
0 & \* & 0 & 0 & 0 \\
0 & 0 & \* & 0 & 0 \\
0 & 0 & 0 & \* & 0 \\
0 & 0 & 0 & 0 & \* 
\end{smallmatrix}
$}}
\right]\\
\bB &=& 
\left[\scalebox{1.05}{\mbox{$
\begin{smallmatrix}
\* \\
\* \\
\* \\
\* \\
\* 
\end{smallmatrix}
$}}
\right] \hspace*{-0.5 mm}, ~
\bC = 
\left[\scalebox{1.05}{\mbox{$
\begin{smallmatrix}
\* & \* & 0 & 0 & 0 \\
0 & \* & \* & 0 & 0 \\
0 & 0 & \* & \* & \* 
\end{smallmatrix}
$}}
\right]
\end{eqnarray*}
\end{subfigure}~\hspace{0.2 mm}
\begin{subfigure}[b]{0.45\textwidth}
\centering
\begin{tikzpicture}[scale=0.65, ->,>=stealth',shorten >=1pt,auto,node distance=1.65cm, main node/.style={circle,draw,font=\scriptsize\bfseries}]
\definecolor{myblue}{RGB}{80,80,160}
\definecolor{almond}{rgb}{0.94, 0.87, 0.8}
\definecolor{bubblegum}{rgb}{0.99, 0.76, 0.8}
\definecolor{columbiablue}{rgb}{0.61, 0.87, 1.0}

  \fill[almond] (-1,-1.5) circle (7.0 pt);
  \fill[almond] (-2,-1.5) circle (7.0 pt);
  \fill[almond] (0,-1.5) circle (7.0 pt);
  \fill[almond] (1,-1.5) circle (7.0 pt);
  \fill[almond] (2,-1.5) circle (7.0 pt);
  \node at (-2,-1.5) {\scriptsize $x_1$};
  \node at (-1,-1.5) {\scriptsize $x_2$};
  \node at (0,-1.5) {\scriptsize $x_3$};
  \node at (1,-1.5) {\scriptsize $x_4$};
  \node at (2,-1.5) {\scriptsize $x_5$};

  \fill[bubblegum] (0,1) circle (7.0 pt);
   \node at (0,1) {\scriptsize $u_1$};
   
  \fill[columbiablue] (-1.5,-3) circle (7.0 pt);
  \fill[columbiablue] (-0.5,-3) circle (7.0 pt);
  \fill[columbiablue] (1.5,-3) circle (7.0 pt);
   \node at (-1.5,-3.0) {\scriptsize $y_1$};
   \node at (-0.5,-3.0) {\scriptsize $y_2$};
   \node at (1.5,-3.0) {\scriptsize $y_3$};
   
 \draw (0,0.75) -> (-2,-1.25);
 \draw (0,0.75) -> (-1,-1.25);
 \draw (0,0.75) -> (0,-1.25);
 \draw (0,0.75) -> (1,-1.25);
 \draw (0,0.75) -> (2,-1.25);
   
  \draw (-2,-1.75) -> (-1.5,-2.75);
  \draw (-1,-1.75) -> (-1.5,-2.75);
  \draw (-1,-1.75) -> (-0.5,-2.75);
  \draw (0,-1.75) -> (-0.5,-2.75);
  \draw (0,-1.75) -> (1.5,-2.75);
  \draw (1,-1.75) -> (1.5,-2.75);
  \draw (2,-1.75) -> (1.5,-2.75);

\path[every node/.style={font=\sffamily\small}]
 
(-2.05,-1.25) edge[loop above] (-2,-1)
(-1.05,-1.25)edge[loop above] (-1,-1)
(0.1,-1.25) edge[loop above]  (0,-1)
(1.0,-1.25) edge[loop above]  (1,-1)
(2.0,-1.25) edge[loop above]  (2,-1);
\end{tikzpicture}
\caption{$\D(\bA, \bB, \bC)$}
\label{fig:digraph_red}
\end{subfigure}
\caption{The structuted system $(\bA,\bB,\bC)$ and digraph $\D(\bA, \bB, \bC)$ for a given universe $\U$ and set $\P$ constructed using Algorithm~\ref{alg:reduc_from_multisetcover} is shown in the above figure.}
\label{fig:eg_multi_to_pb1}
\end{figure}

The pseudo-code showing a reduction of Problem~\ref{prob:design} to an instance of MSMC problem is given in   Algorithm~\ref{alg:reduc_from_multisetcover}.  Given a general instance of the MSMC problem consisting of universe $\U=\{1,\ldots,N\}$, sets $\P=\{\S_1,\ldots,\S_r\}$ and constant demand $\alpha$, we construct a structured system $(\bA,\bB,\bC)$ with states $x_1,\ldots,x_N$, input $u_1$, and outputs $y_1,\ldots,y_p$ (Step~\ref{step:state_inp_oup}). In $\bA$, every diagonal entry is $\*$ (Step~\ref{step:state_conn}). Thus the system is structurally cyclic. Moreover, $\B(\bA)$ has a perfect matching $M=\{(x'_i,x_i)$ for $i \in \{1,\ldots,N\} \}$. The input matrix $\bB$ consists of  a single input $u_1$ which is connected to every state node (Step~\ref{step:inp_conn}). The output matrix $\bC$ is constructed depending on $\P$ such that an output node $y_i$ senses all state nodes $x_j$'s that satisfy $j \in \S_i$ (Step~\ref{step:oup_conn}). An illustrative example demonstrating the  construction of the structured system $(\bA, \bB, \bC)$ for a given instance of the MSMC problem is given in Figure~\ref{fig:eg_multi_to_pb1}. Note that, the value of $\gamma$ for Problem~\ref{prob:verification} for the constructed structured system is uniquely defined by the value  of $\alpha$ of the corresponding MSMC problem.  Using the structured system $(\bA, \bB, \bC)$ constructed in Algorithm~\ref{alg:reduc_from_multisetcover}, we have the following result. 

\begin{lemma}\label{lem:lem_NP}
Consider the MSMC problem $(\U,\P,\alpha)$ and the structured system $(\bA,\bB,\bC)$ constructed using Algorithm~\ref{alg:reduc_from_multisetcover}. Then, $\bK$ is a solution to Problem~\ref{prob:design} if and only if the sets selected under $\bK$, $\S(\bK)$ is a solution to Problem~\ref{prob:MSMC}.
\end{lemma}
\begin{proof}
\textbf{Only-if part:} Here we assume $\bK \in \K_\gamma$ and prove that $\S(\bK)$ covers each element in $\U$ at least $\alpha$ times, i.e.,  $\S(\bK)$ is a multi-cover of  the universe $\U$ satisfying demand $\alpha$. We prove this using contradiction. Suppose there exists an element $j \in \U$ that is not covered $\alpha$ times by $\S(\bK)$. Let $\S(\bK)$ consists of sets $\S_{i_1},\ldots,\S_{i_k}$ and the corresponding outputs are $y_{i_1},\ldots,y_{i_k}$. From the construction (Step~\ref{step:oup_conn} of Algorithm~\ref{alg:reduc_from_multisetcover}), an output node $y_i$ has incoming edges from state nodes $x_j$'s for all $j \in \S_i$. As element $j$ appears less than $\gamma+1$ times, since $\alpha = \gamma+1$, in the union of sets $\S_{i_1},\ldots,\S_{i_k}$, corresponding state node $x_j$ there are less than $\gamma+1$ outgoing edges to outputs.  Without loss generality, assume that corresponding state node $x_j$ there are  $\gamma$ outgoing edges to outputs. Notice that, as $\bA$ is a diagonal matrix $x_j$ is not connected to any other state node other than itself. Thus $x_j$ lies in an SCC with $\gamma$ feedback links. Then, the closed-loop system will have SFMs when  $\gamma$ feedback links fail. This contradicts the fact that $\bK$ is a solution to Problem~\ref{prob:design}. This completes the only-if part.

\noindent \textbf{If part:} Here we assume that $\S(\bK)$ is a solution to Problem~\ref{prob:MSMC} and prove that $\bK$ is a solution to Problem~\ref{prob:design}. Suppose not. From the construction, an output node $y_i$ has incoming edges from state nodes $x_j$'s for all $j \in \S_i$. Since $\B(\bA)$ has a perfect matching (as $\bA$ is diagonal), condition~b) of Proposition~\ref{prop:SFM1} is satisfied without using any feedback edge. Thus, the assumption $\bK \notin \K_\gamma$  implies that $\bK$ violates condition~a) in Proposition~\ref{prop:SFM1} after removing some subset of feedback links which has cardinality at most $\gamma$. Without loss of generality, assume that the subset of feedback links have cardinality $\gamma$. In other words, there exists a state node $x_j$  which does not lie in an SCC in $\D(\bA, \bB, \bC, \bK)$ with more than $\gamma$ feedback links. As $\bA$ is a diagonal matrix, state $x_j$ is not connected to any other state node and the feedback links corresponding to the SCC  in which  $x_j$ lies are those which correspond to the outputs connected to $x_j$. In $\D(\bA, \bB, \bC, \bK)$, $x_j$  has less than $\gamma+1$ outputs connected to it. This implies that the corresponding element $j \in \U$ lies in less than $\gamma+1$ sets in $\S(\bK)$. This is a contradiction to the assumption that $\S(\bK)$ is a multi-cover of universe $\U$ with demand $\alpha$, as $\alpha = \gamma+1$. This completes the proof.
\end{proof}

Next, we discuss the NP-hardness of Problem~\ref{prob:design} using a reduction from MSMC Problem. We show that any instance of MSMC Problem can be reduced to an instance of Problem~\ref{prob:design} such that an optimal solution to Problem~\ref{prob:design} gives an optimal solution to the MSMC Problem. 

\begin{theorem}\label{th:prob1_NP}
Consider the MSMC problem $(\U,\P,\alpha)$ and let $(\bA,\bB,\bC)$ be the structured system constructed using Algorithm~\ref{alg:reduc_from_multisetcover}. Let $\bK^{\*}$ be an optimal solution to Problem~\ref{prob:design} and $\S(\bK^{\*})$ be the set multi-cover corresponding to $\bK^{\*}$. Then, (i)~$\S(\bK^{\*})$ is feasible and an optimal solution to the MSMC problem, and
(ii)~Problem~\ref{prob:design} is NP-hard for structurally cyclic systems.
\end{theorem}
\begin{proof}
{\bf (i)}~Given a general instance of the MSMC problem $(\U,\P,\alpha)$, we reduce it to an instance of Problem~\ref{prob:design} using Algorithm~\ref{alg:reduc_from_multisetcover}. Let $\bK$ be a feasible solution to Problem~\ref{prob:design}. Using Lemma~\ref{lem:lem_NP} the sets selected under $\bK$, i.e., $\S(\bK)$, covers each element in $\U$ at least $\alpha$ times. Hence $\S(\bK)$ is a feasible solution to set multi-covering problem. Next, we prove optimality. 

Let $\bK^{\*}$ be an optimal solution to Problem~\ref{prob:design}. Using Lemma~\ref{lem:lem_NP}, $\S(\bK^{\*})$ is a feasible solution to set multi-covering problem. From Step~\ref{step:sol_to_msp} of Algorithm~\ref{alg:reduc_from_multisetcover} we know that $\norm[\bK^{\*}]_0=|\S(\bK^{\*})|$, where $|D|$ denoted cardinality of set $D$.  Using contradiction, we now prove that $\S(\bK^{\*})$ is an optimal solution to the MSMC problem. Assume $\S(\bK^{\*})$ is not an optimal solution. Then there exists a solution $\widetilde{\S} \subseteq \P$ such that $|\widetilde{\S}| < |\S(\bK^{\*})|$ and satisfies $\cup_{\S_i \in \widetilde{\S}} = \U$ and each element in $\U$ is covered at least $\alpha$ times. With respect to $\widetilde{\S}$, define $\widetilde{\bK}$, where $\widetilde{\bK}_{ij}=\*$ if $\S_j \in \widetilde{\S}$. 
 From Lemma~\ref{lem:lem_NP}, $\tilde{\bK}$ is a solution to Problem~\ref{prob:design}. Further, $|\widetilde{\S}| < |\S(\bK^{\*})|$ implies $\norm[\widetilde{\bK}]_0 < \norm[\bK^{\*}]_0$. This is a contradiction as $\bk^{\*}$ is an optimal solution to Problem~\ref{prob:design}. This proves that $\S(\bK^\*)$ is an optimal solution to Problem~\ref{prob:design}.

\noindent{\bf (ii)}~To prove (ii), we first give the computational complexity of Algorithm~\ref{alg:reduc_from_multisetcover}. Steps~\ref{step:state_conn}, \ref{step:inp_conn} and \ref{step:sol_to_msp} has complexity $O(N)$, and Step~\ref{step:oup_conn} has complexity $O(Nr)$. Here, $N=n$ and $r=p=O(n)$. Hence total complexity of Algorithm~\ref{alg:reduc_from_multisetcover} is $O(n^2)$. Thus Algorithm~\ref{alg:reduc_from_multisetcover} gives a polynomial time reduction of minimum set multi-covering problem to an instance of Problem~\ref{prob:design} such that optimal solution to Problem~\ref{prob:design} gives an optimal solution to MSMC problem (from Theorem~\ref{th:prob1_NP}). Moreover, the structured system $\bA$ constructed in Algorithm~\ref{alg:reduc_from_multisetcover} is structurally cyclic (all state nodes are  spanned by disjoint cycles). Hence Problem~\ref{prob:design} is NP-hard for structurally cyclic systems. This completes the proof.
\end{proof}
Now we present the inapproximability\footnote{A $\Delta$-optimal approximation algorithm is an algorithm whose solution value is at most $\Delta$ times that of the actual optimum value.} result of Problem~\ref{prob:design}. 
\begin{theorem}\label{th:approxi_NP_sol}
Consider the MSMC problem $(\U,\P,\alpha)$ and the structured system $(\bA,\bB,\bC)$  constructed using Algorithm~\ref{alg:reduc_from_multisetcover}. Then, (i)~for any $\epsilon \geqslant 1$, if there exists a $\epsilon$-optimal solution to Problem~\ref{prob:design}, then there exists a $\epsilon$-optimal solution to the MSMC problem, (ii)~Problem~\ref{prob:design} is inapproximable to $(1-o(1))\,$ $\log~n$, where $n$ denotes the system dimension.
\end{theorem}
\begin{proof}
{\bf (i)}~Recall that $\bK^\*$ is an optimal solution to Problem~\ref{prob:design} and let  $\S^\*$ be an optimal solution to Problem~\ref{prob:MSMC}. To prove~(i), we need to show that if $\bK'\in \K_\gamma$ and $\norm[\bK']_0 \leqslant \epsilon\,\norm[\bK^{\*}]_0$, then $|\S(\bK')| \leqslant \epsilon\,|\S^\*|$ . Note that $\norm[\bK']_0 = |\S(\bK')|$ and $\norm[\bK^\*]_0 = |\S(\bK^\*)|$ (see Step~\ref{step:sol_to_msp}). Thus $|\S(\bK')| \leqslant \epsilon\,|\S(\bK^\*)|$. By Theorem~\ref{th:prob1_NP}~(i), $\S(\bK^\*)$ is an optimal solution to Problem~\ref{prob:MSMC}. This implies $|\S(\bK^\*)| = |\S^\*|$. Thus  $|\S(\bK')| \leqslant \epsilon\,|\S^\*|$. 

\noindent{\bf(ii)}~From Theorem~\ref{th:approxi_NP_sol}~(i),  for any $\epsilon \geqslant 1$, if there exists an $\epsilon$-optimal solution to Problem~\ref{prob:design} then there exists an $\epsilon$-optimal solution to the MSMC problem.  The MSMC problem is inapproximable to factor $(1-o(1))\,\log\,n$, as the minimum set cover problem \cite{Chv:79}, \cite{Fei:98} is a special case of MSMC problem for $\alpha=1$. Thus Problem~\ref{prob:design} is inapproximable to $(1-o(1))\,\log\,n$.
This completes the proof. 
\end{proof}

 We conclude that Problem~\ref{prob:design} is NP-hard and also inapproximable to factor $(1-o(1))\,\log\,n$ for general systems as well as structurally cyclic systems.  Table~\ref{tb:NP} summarizes the complexity results obtained in this paper for different cases of Problem~\ref{prob:verification} and Problem~\ref{prob:design}.
\begin{table*}[h] 
\centering
\captionof{table}{Algorithmic complexity results of the optimal feedback selection problem}\label{tb:NP}
\begin{tabular}{|l|c|c|c|}
\hline
{\diagbox[width=12em]{Graph topology}{Problem}}
 & { Problem~\ref{prob:verification}~} & {Problem~\ref{prob:design}~}   \\
\cline{2-3}
\hline
\multicolumn{1}{|l|}{\vphantom{$9^{9^9}$} Irreducible $\D(\bA)$} &  \shortstack{\phantom{$8^{7^7}$}NP-complete}~ & \shortstack{NP-hard}    \\
\hline
\multicolumn{1}{|l|}{\vphantom{$6^{6^6}$}Structurally cyclic systems} &  \shortstack{\vphantom{$9^{9^9}$}P}~ & \shortstack{ NP-hard}  \\
\hline
\end{tabular}
\end{table*}
In the next section, we provide an approximation algorithm to solve Problem~\ref{prob:design} for a special graph topology of practical importance.

\section{Approximation Algorithm for Sparsest Resilient Feedback  Design Problem for Back-Edge Feedback Structure}\label{sec:algo_design}
In this section, we propose an {\em order optimal}, $O(\log\,n)$,  approximation algorithm for the sparsest  resilient feedback design problem for structurally cyclic systems with a special feedback structure so-called {\em back-edge} feedback structure.
We assume that the feedback matrix satisfies a structural constraint that all feedback edges $(y_j,u_i)'$s are such that there exists a directed path from input $u_i$ to output $y_j$ in $\D(\bA,\bB,\bC)$. In other words,  an output from a state is fed back to an input which can directly or indirectly influence the state associated with that output.  We refer to a feedback structure that satisfies this constraint  as  {\em back-edge} feedback structure.   Back-edge feedback structure is applicable in various networks that include hierarchical networks and multi-agent systems in which the state measurement is fed back to the leader agent. The hierarchical network structure is common in real-life networks \cite{LiuYanSlo:12}. A power distribution system follows a hierarchical network structure and finding an optimal resilient control strategy aims towards designing a least cost feedback pattern to  maintain the system parameters such as voltages and frequency at different layers of the network at specified levels  even under adversarial conditions \cite{FegPer:77}, \cite{Mar:07}. There is a wide class of practically important systems called  {\it self-damped} systems \cite{ChaMes:13} that are structurally cyclic, for example, consensus dynamics in multi-agent systems and epidemic dynamics. 

For structurally cyclic systems with back-edge feedback structure, we propose a polynomial time algorithm to find an approximate solution to Problem~\ref{prob:design} with an optimal approximation ratio. We describe below the graph topology considered in this section.

Consider  a digraph $\D_{\G} = (V_{\G},E_{\G})$ and let  $v_i \in V_{\G}$ and $v_j \in V_{\G}$ be such that there exists a directed path from $v_i$ to $v_j$. Then  $v_i$ is referred to as an {\em ancestor} of $v_j$ and $v_j$ as a {\em descendant} of $v_i$. Now we give the following assumption.

\begin{assume}\label{asm:special_graph_topo}
Consider a structured system $(\bA, \bB, \bC)$. The feedback matrices for $(\bA, \bB, \bC)$ are constrained in such a way that a feedback edge $(y_j, u_i)$ is feasible only if $u_i$ is an ancestor of $y_j$  in the digraph $\D(\bA, \bB,\bC)$.
\end{assume}

Consider the structured system $(\bA, \bB, \bC)$ constructed using Algorithm~\ref{alg:reduc_from_multisetcover}. The feedback matrix for the constructed $(\bA, \bB, \bC)$ satisfies Assumption~\ref{asm:special_graph_topo}, i.e., feedback edge $(y_j, u_i)$ is feasible only if $u_i$ is an ancestor of $y_j$  in the digraph $\D(\bA, \bB,\bC)$. Moreover, all the state nodes in the  structured system are spanned by disjoint union of cycles and hence the system is structurally cyclic. A schematic diagram that shows this special feedback structure is given in Figure~\ref{fig:eg_multi_to_pb1}.  The following theorem is an immediate consequence of Theorem~\ref{th:prob1_NP}.

\begin{cor}\label{cor:NP}
Consider a structurally cyclic  structured system $(\bA, \bB, \bC)$. Let Assumption~\ref{asm:special_graph_topo} holds. Then, Problem~\ref{prob:design} is NP-hard for structured systems with this  graph topology.
\end{cor}

Now we propose an approximation algorithm for solving Problem~\ref{prob:design} for a structured system $(\bA, \bB, \bC)$ under Assumption~\ref{asm:special_graph_topo}.  First, we propose an algorithm to reduce the general instance of Problem~\ref{prob:design} for a structurally system satisfying Assumption~\ref{asm:special_graph_topo} to an instance of minimum set multi-covering problem.

\begin{algorithm}[t]
  \caption{Psuedo-code for reducing Problem~\ref{prob:design} to an instance of MSMC problem
  \label{alg:reduc_from_prob1}}
  \begin{algorithmic}
  	\State \textit {\bf Input:} Structured system $(\bA, \bB, \bC)$ and a constant $\gamma$
	\State \textit{\bf Output:}  MSMC problem $(\U,\P,\alpha)$
  \end{algorithmic}
  \begin{algorithmic}[1]
  	\State Let $\{x_1,\ldots,x_n\}$ be the state nodes, $\{u_1,\ldots,u_m\}$ be input nodes, and $\{y_1,\ldots,y_p\}$ be output nodes \label{step:state_inp_oup_2}
  	\State Define $\bK:=\{\bK_{ij}=\*: u_i$ is ancestor of $y_j$ in $\D(\bA,\bB,\bC)\}$\label{step:bK_def}
  	\State Let $E_K$ be the set of feedback edges corresponding to $\bK$ \label{step:Ek_def}
  	\State Define universe $\U \leftarrow \{x_1,\ldots,x_n\}$ \label{step:univ_def}
  \State Define set $\P~\leftarrow~\{\S_1,\ldots,\S_{|E_K|}\}$ \label{step:P_sets_def}
  \For{$e_d = (y_j,u_i) \in E_K$}
    \State $\S_d~:=~\{x_a: x_a$ lies in some directed path  from $u_i$ to $y_j\}$ \label{step:sets_def}
 \EndFor\label{step:endfor}
    \State $\alpha \leftarrow \gamma+1$
 \State Given a solution $\S'$ to the MSMC problem for  $(\U,\P,\alpha)$, define feedback matrix selected under $\S'$, $\bK(\S'):= \{\bK(\S')_{ij}=\*:\S_d \in \S'~{\rm and}~e_d=(y_j,u_i)\}$ \label{step:multi_sol}
  \end{algorithmic}
\end{algorithm}

Given a general instance of the structured system $(\bA,\bB,\bC)$ satisfying Assumption~\ref{asm:special_graph_topo} and a constant $\gamma$, we construct an instance of the MSMC problem $(\U,\P,\alpha)$ using Algorithm~\ref{alg:reduc_from_prob1}. Let $\{x_1,\ldots,x_n\}$ be the state nodes, $\{u_1,\ldots,u_m\}$ be the input nodes, and $\{y_1,\ldots,y_p\}$ be the output nodes of the structured system (Step~\ref{step:state_inp_oup_2}).  Define $\bK$ such that $\bK_{ij}=\*$ only if there exists a directed path from $u_i$ to $y_j$ in $\D(\bA, \bB, \bC)$ (Step~\ref{step:bK_def}).   Thus the edge set $E_K$  corresponding to $\bK$ consists of all feedback edges  that satisfy Assumption~\ref{asm:special_graph_topo} (Step~\ref{step:Ek_def}).
Now we define the universe $\U$ such that it contains all state nodes $\{x_1,\ldots,x_n\}$ (Step~\ref{step:univ_def}). The set $\P$ consists of sets $\S_1,\ldots,\S_{|E_K|}$, where set $\S_d$ corresponds to a feedback edge  $e_d = (y_j, u_i)$ (Step~\ref{step:P_sets_def}). The set $\S_d$ consists of those elements in the universe that corresponds to state nodes  in $\D(\bA, \bB, \bC)$ that lie in some directed path from $u_i$ to $y_j$ (Step~\ref{step:sets_def}). Note that there may be multiple paths from $u_i$ to $y_j$ as shown in Figure~\ref{fig:eg_multi_to_pb2}. For a solution $\S'$ to the MSMC problem for $(\U,\P,\alpha)$, we define feedback matrix $\bK(\S')$ (Step~\ref{step:multi_sol}). Here $\bK(\S')$ consists of all those feedback edges that correspond to sets in $\S'$ under the definition given in Step~\ref{step:sets_def}. An illustrative example demonstrating the  construction of MSMC problem  for a given instance of a structured system $(\bA, \bB, \bC)$ is given in Figure~\ref{fig:eg_multi_to_pb2}.

\begin{figure}[t]
\centering
\begin{subfigure}[b]{0.45\textwidth}
\centering
\begin{tikzpicture}[scale=0.65, ->,>=stealth',shorten >=1pt,auto,node distance=1.65cm, main node/.style={circle,draw,font=\scriptsize\bfseries}]
\definecolor{myblue}{RGB}{80,80,160}
\definecolor{almond}{rgb}{0.94, 0.87, 0.8}
\definecolor{bubblegum}{rgb}{0.99, 0.76, 0.8}
\definecolor{columbiablue}{rgb}{0.61, 0.87, 1.0}

\fill[bubblegum] (-1,1) circle (7.0 pt);
\fill[bubblegum] (0,1) circle (7.0 pt);
\fill[bubblegum] (1,1) circle (7.0 pt);
\node at (-1,1) {\scriptsize $u_1$};
\node at (0,1) {\scriptsize $u_2$};
\node at (1,1) {\scriptsize $u_3$};  

\fill[almond] (-1,-0.5) circle (7.0 pt);
\fill[almond] (0,-0.5) circle (7.0 pt);
\fill[almond] (1,-0.5) circle (7.0 pt);
\node at (-1,-0.5) {\scriptsize $x_1$};
\node at (0,-0.5) {\scriptsize $x_3$};
\node at (1,-0.5) {\scriptsize $x_5$};
  
\fill[almond] (-1,-2) circle (7.0 pt);
\fill[almond] (0,-2) circle (7.0 pt);
\fill[almond] (1,-2) circle (7.0 pt);
\node at (-1,-2) {\scriptsize $x_2$};
\node at (0,-2) {\scriptsize $x_4$};
\node at (1,-2) {\scriptsize $x_6$};


\fill[columbiablue] (-1,-3.5) circle (7.0 pt);
\fill[columbiablue] (0,-3.5) circle (7.0 pt);
\fill[columbiablue] (1,-3.5) circle (7.0 pt);
\node at (-1,-3.5) {\scriptsize $y_1$};
\node at (0,-3.5) {\scriptsize $y_2$};
\node at (1,-3.5) {\scriptsize $y_3$};

\draw (-1,0.75) -> (-1,-0.25);
\draw (0,0.75) -> (0,-0.25);
\draw (1,0.75) -> (1,-0.25);

\draw (-1,-0.75) -> (-1,-1.75);
\draw (0,-0.75) -> (0,-1.75);
\draw (1,-0.75) -> (1,-1.75);

\draw (1,-2.25) -> (1,-3.25); 
\draw (0.25,-2) -> (0.75,-2); 

\draw (-1,-2.25) -> (-1,-3.25);
\draw (0,-2.25) -> (0,-3.25);
\draw (1,-2.25) -> (1,-3.25);
\draw (0,-0.75) -> (-1,-3.25);
\draw (0,0.75) -> (-1,-1.75);

\path[every node/.style={font=\sffamily\small}]
(-0.9,-0.25)edge[loop above] (-1,-1)
(0.1,-0.25)edge[loop above] (0,-1)
(1.1,-0.25)edge[loop above] (1,-1)
(-0.9,-1.75)edge[loop above] (-1,-2.5)
(0.1,-1.75)edge[loop above] (0,-2.5)
(1.1,-1.75)edge[loop above] (1,-2.5);
\end{tikzpicture}
\caption{$\D(\bA, \bB, \bC)$}
\label{fig:digraph_r}
\end{subfigure}
\begin{subfigure}[b]{0.45\textwidth}
\begin{eqnarray*} \label{eq:sysone}
 \U &=& \{x_1,x_2,x_3,x_4,x_5,x_6\},\\
\S_1 &=& \{x_1,x_2\},~\S_2=\{x_3\},\\
\S_3 &=& \{x_3,x_4\},~\S_4=\{x_5,x_6\},\\
 \S_5 &=& \{x_3,x_4,x_6\},~\S_6=\{x_2\}
\end{eqnarray*}
 \caption{MSMC problem}
\end{subfigure}~\hspace{0.2 mm}
\caption{The MSMC problem instance $(\U,\P,\alpha)$ constructed for a structured system digraph $\D(\bA, \bB, \bC)$ using Algorithm~\ref{alg:reduc_from_prob1} is shown in the figure~(a) and (b).}
\label{fig:eg_multi_to_pb2}
\end{figure} 

\begin{lemma}\label{lem:lem1_multi_approx}
Consider a structurally cyclic structured system $(\bA,\bB,\bC)$   and a constant $\gamma$. Let Assumption~\ref{asm:special_graph_topo} holds and $(\U,\P,\alpha)$ be the MSMC problem constructed using Algorithm~\ref{alg:reduc_from_prob1}. Then, $\S'$ is a solution to the MSMC problem if and only if the feedback matrix $\bK(\S')$ selected under $\S'$ is a solution to Problem~\ref{prob:design}.
\end{lemma}
\begin{proof}
As the system is structurally cyclic condition~b) in Proposition~\ref{prop:SFM1} is satisfied without using any feedback edges and only condition~a) has to be satisfied.

\textbf{Only-if part:} Here we assume that $\S'$ is a solution to the MSMC problem and prove that $\bK(\S')$ is a solution to Problem~\ref{prob:design}. We prove this using contradiction. Assume that $\bK(\S')$ is not a solution to Problem~\ref{prob:design}. Since $\B(\bA)$ has a perfect matching, condition~b) of Proposition~\ref{prop:SFM1} is satisfied without using any feedback edge. Thus $\bK(\S')$ must violate condition~a) after removing some $\gamma$ feedback links. There exists a state node, say $x_q$, that lies in an SCC in $\D(\bA, \bB, \bC, \bK(\S'))$ with less than $\gamma+1$ feedback links. Notice that $\bK(\S')$ consists of feedback edges corresponding to all the sets in $\S'$. From the construction (Step~\ref{step:sets_def}), set $\S_d$ consists of all state nodes which lie in some directed path from $u_i$ to $y_j$ in $\D(\bA, \bB, \bC)$.   Thus in $\D(\bA, \bB, \bC)$ there are less than $\gamma$ different directed paths from some input node to some output node through $x_q$. Hence the element $x_q$ is covered less than $\alpha$ times by the cover $\S'$ as $\alpha=\gamma+1$. This is a contradiction as $\S'$ is a solution to the MSMC problem. This completes the only-if part.

\noindent \textbf{If part:} Here we assume that $\bK(\S')$ is a solution to Problem~\ref{prob:design} and prove that $\S'$ is a solution to the MSMC problem. We prove this using contradiction. Assume $\S'$ is not a solution to  Problem~\ref{prob:MSMC}. Then there exists an element $x_q$ which is covered less than $\gamma+1$ times by the cover $\S'$ (since $\alpha=\gamma+1$). So $\S'$ consists of less than $\gamma+1$ sets which contain $x_q$. By the construction of sets $\S_d$ (Step~\ref{step:sets_def}), $\S_d$ consists of all state nodes that lie in some directed path in $\D(\bA, \bB, \bC)$ from input $u_i$ to  output $y_j$ for $e_d=(y_j, u_i)$. So in $\D(\bA, \bB, \bC)$ there are less than $\gamma+1$ different directed paths from some input node to some output node through $x_q$. Notice that $\bK(\S')$ consists of feedback edges corresponding to all the sets in $\S'$. Thus state node $x_q$ lies in an SCC in $\D(\bA, \bB, \bC, \bK(\S'))$ with less than $\gamma+1$ feedback links. So after removing $\gamma$ feedback links, condition~a) of proposition~\ref{prop:SFM1} is violated. This is a contradiction as $\bK(\S')$ is a solution to Problem~\ref{prob:design}. This completes the if-part.
\end{proof}

\begin{theorem}\label{th:multi_opti}
Consider a structurally cyclic structured system $(\bA,\bB,\bC)$ and a constant $\gamma$. Let Assumption~\ref{asm:special_graph_topo} holds and $(\U,\P,\alpha)$ be the MSMC problem constructed using Algorithm~\ref{alg:reduc_from_prob1}. Let $\S^{\*}$ be an optimal solution to MSMC problem and  $\bK(\S^{\*})$ be the feedback matrix selected under $\S^{\*}$. Then, 
\begin{itemize}
\item[(i)] $\bK(\S^{\*})$ is an optimal solution to Problem~\ref{prob:design}, and
 
\item[(ii)] For every $\epsilon \geqslant 1$, if there exists a $\epsilon$-optimal solution to MSMC problem, then there exists a $\epsilon$-optimal solution to the Problem~\ref{prob:design}, i.e., $|\S'| \leqslant \epsilon \, |\S^\*|$ implies $\norm[\bK(\S')]_0 \leqslant\epsilon \,\norm[\bK^\*]_0$, where $\S^\*$ is an optimal solution to MSMC problem and $\bK^\*$ is an optimal solution to Problem~\ref{prob:design}.
\end{itemize}
\end{theorem}
\begin{proof}
{\bf (i)}~Given $\S^\*$ is an optimal solution to Problem~\ref{prob:MSMC}. By Lemma~\ref{lem:lem1_multi_approx}, $\bK(\S^{\*})$ is a feasible solution to Problem~\ref{prob:design}. Now we prove optimality of $\bK(\S^{\*})$ using contradiction. Assume that $\bK(\S^{\*})$ is not an optimal solution to Problem~\ref{prob:design}. Then there exists a solution to Problem~\ref{prob:design}, say $\bK^1$,  such that $\norm[\bK^1]_0 < \norm[\bK(\S^{\*})]_0$. Consider $\S^1 := \{\S_j:\S_j$ consists of  $x_q$'s which lie in some directed path from $u_i$ to $y_j$ if $\bK_{ij}^1=\*\}$.
Here, $\norm[\bK^1]_0 = |\S^1|$. From Lemma~\ref{lem:lem1_multi_approx}, $\S^1$ is a feasible solution to MSMC problem. As $\norm[\bK^1]_0 < \norm[\bK(\S^{\*})]_0$, $|\S^1|<|\S^{\*}|$. This is a contradiction as $\S^{\*}$ is an optimal solution. Hence $\bK(\S^\*)$ is an optimal  solution to Problem~\ref{prob:design}.

\noindent{\bf (ii)}~ Let $\S^{\*}$ be an optimal solution to the MSMC problem.  Given $|\S'| \leqslant \epsilon |\S^{\*}|$ and we need to show that $\norm[\bK(\S')]_0 \leqslant \epsilon \norm[\bK^{\*}]_0$. From Step~\ref{step:multi_sol} of Algorithm~\ref{alg:reduc_from_prob1} we get $\norm[\bK(\S^{\*})]_0=|\S^{\*}|$. As $|\S'| \leqslant \epsilon \, |\S^{\*}|$, $\norm[\bK(\S')]_0=|\S'| \leqslant \epsilon\, |\S^{\*}| = \epsilon\, \norm[\bK(\S^{\*})]_0$. Moreover, from Theorem~\ref{th:multi_opti}~(i), $\bK(\S^{\*})$ selected under $\S^{\*}$ is an optimal solution to Problem~\ref{prob:design}.  Hence $\norm[\bK(\S^{\*})]_0 = \norm[\bK^{\*}]_0$ and $\norm[\bK(\S')]_0 \leqslant\epsilon \,\norm[\bK^\*]_0$. This completes the proof.  
\end{proof}

Theorem~\ref{th:multi_opti} proves that if one can find an approximate solution to the MSMC problem, then it gives an approximate solution to Problem~\ref{prob:design}.
 
\begin{theorem}\label{th:th_multi_approx}
Consider a structurally cyclic structured system $(\bA,\bB,\bC)$ and a constant $\gamma$. Let Assumption~\ref{asm:special_graph_topo} holds and $(\U,\P,\alpha)$ be the MSMC problem constructed using Algorithm~\ref{alg:reduc_from_prob1}. Then, ~Problem~\ref{prob:design} is approximable to $O(\log~ n)$ in polynomial time.
\end{theorem}
\begin{proof}
To prove this, we first give the complexity of Algorithm~\ref{alg:reduc_from_prob1}. Finding a directed path between an input and an output has $O(n)$ complexity. The possible number of paths is at most $mp$ and $m=p=O(n)$. Thus the complexity of Step~\ref{step:bK_def} is $O(n^3)$.  Step~\ref{step:univ_def} has complexity $O(n)$. Steps~\ref{step:P_sets_def}-\ref{step:endfor} are of complexity $O(|E_K|)$ which is $O(mp)$. In general $m=p=O(n)$.  Remaining steps are of linear complexity.  Thus Steps~\ref{step:state_inp_oup_2}-\ref{step:endfor} have complexity $O(n^3)$. Thus Algorithm~\ref{alg:reduc_from_prob1}   gives a polynomial time reduction of Problem~\ref{prob:design} for a structured system that satisfies Assumption~\ref{asm:special_graph_topo} to an instance of the MSMC problem in $O(n^3)$ operations. 

From Theorem~\ref{th:multi_opti}~(ii), an $\epsilon$-optimal solution to MSMC problem gives an $\epsilon$-optimal solution to Problem~\ref{prob:design}, for any $\epsilon \geqslant 1$. There exists a polynomial time algorithm to solve MSMC problem which gives $O(\log~n)$-optimal solution \cite[Theorem~5.2]{RajVaz-93}. Now using Algorithm~\ref{alg:reduc_from_prob1} in this paper and Algorithm~5.2 given in \cite{RajVaz-93},  Problem~\ref{prob:design} is approximable to $O(\log~n)$ in polynomial time.
\end{proof}
Theorem~\ref{th:th_multi_approx} thus concludes that Problem~\ref{prob:design} is approximable to factor $O(\log\,n)$ for structured systems satisfying Assumption~\ref{asm:special_graph_topo} in polynomial time.

In the next section, we provide an algorithm to solve Problem~\ref{prob:verification} when $D(\bA)$ is irreducible.


\section{Algorithm for  Feedback Resilient Verification Problem}\label{sec:algo_verify}
In this section, we propose an algorithm to solve Problem~\ref{prob:verification} for irreducible systems.  Note that, Problem~\ref{prob:verification} is NP-complete even for irreducible systems (Corollary~\ref{cor:NP_irreducible}). The proposed algorithm is computationally efficient  for smaller values of $\gamma$.  Typically, while attacking cyber-physical systems the attacker targets few links due to resource and infrastructure constraints and to remain undetected by the system.  
Henceforth, the following assumption holds.
\begin{assume}\label{asm:irr}
For the structured system $(\bA, \bB, \bC, \bK)$, the digraph $\D(\bA)$ is irreducible.
\end{assume}
If $\D(\bA)$ is irreducible, then only one feedback link is enough to satisfy condition~a) in Proposition~\ref{prop:SFM1}. Thus the class of irreducible systems satisfies the no-SFMs criteria if the system bipartite graph $\B(\bA, \bB, \bC, \bK)$ has a perfect matching and $\D(\bA, \bB, \bC, \bK)$ has at least one feedback link present in it. Hence, Problem~\ref{prob:verification} for irreducible systems boils down to checking the existence of a perfect matching in $\B(\bA, \bB, \bC, \bK)$ after the failure of any $\gamma$ feedback links.
However, if there exists a perfect matching in $\B(\bA, \bB, \bC, \bK)$ that uses no feedback links, then condition~b) in Proposition~\ref{prop:SFM1} is satisfied without using any feedback edge: in such a case, any one feedback edge is sufficient to satisfy the no-SFMs criteria and the system is resilient for any $\gamma < |\E_K|$. Thus, the case where all perfect matchings in $\B(\bA, \bB, \bC, \bK)$ have at least one feedback edge is of interest  and considered here.

In this section, we discuss our approach to solve Problem~\ref{prob:verification}. If a system is resilient to failure of any $\gamma$ feedback links, then it is resilient to failure of any set of feedback links of cardinality less than $\gamma$. Thus,  to solve Problem~\ref{prob:verification} it is enough to verify if the system is resilient to failure of any $\gamma$ feedback links. 
Problem~\ref{prob:verification} is NP-complete for a general $\gamma$. As $m=O(n)$ and $p=O(n)$, the number of feedback links is $O(n^2)$. So there are $\binom{n^2}{\gamma}$ bipartite graphs possible for removal of any $\gamma$ feedback links. The exhaustive search-based technique requires checking perfect matching in $\binom{n^2}{\gamma}$ bipartite graphs which is exponential in $n^2$ and the complexity is huge even for small $\gamma$ as $n$ is large for complex systems. We now present an algorithm to solve Problem~\ref{prob:verification} with a significant saving in computations when compared to the exhaustive search-based approach. 

\subsection{Algorithm and Results for $\gamma=1$}\label{subsec:k=1}
This subsection elaborates an algorithm for solving  Problem~\ref{prob:verification} when the possible number of breakages are at most one, i.e., $\gamma=1$. An exhaustive search-based technique to solve this case involves  removing each of the feedback edges individually from $\E_K$ and then checking for the existence of perfect matching in the new bipartite graph for each of this case. Note that, the number of feedback links is $O(n^2)$: exhaustive search-based technique requires verifying $O(n^2)$ cases. We present a scheme where only $O(n)$ cases have to be verified.

\begin{algorithm}[t]
\caption{Psuedocode to verify resilience of bipartite graph $\B_1=(V_1, \widetilde{V}_1, \E_1)$  for any one edge removal from set $S_1$\label{alg:k=1}}
\begin{algorithmic}
\State \textit {\bf Input:} Bipartite graph $\B_1=(V_1, \widetilde{V}_1, \E_1)$ and  set $S_1 \subset \E_1$
\State \textit{\bf Output:} True or False 
\end{algorithmic}
\begin{algorithmic}[1]
\State Define cost function, $c_0(e) \leftarrow \begin{cases}
  1, {\rm for~} e \in S_1,\\
  0, {\rm otherwise}.
\end{cases}  \label{step:cost1}
  $
  \State Find min-cost perfect matching in $\B_1$ using  $c_0$, say $M_{0}$ \label{step:MCMM}
  \State Define $F_{0} := M_0 \cap  S_1$, where $F_{0}=\{f_1,\ldots,f_{\ell}\}$\label{step:F0}
 \State Define $result$ $\leftarrow$ True \label{step:res_1}
 \For{$i=1$ to $\ell$, $i++$}
    \State  Check  for existence of perfect matching in $\B_1$ after removing edge $f_i$ \label{step:remove_fi}
    \If {matching does not exist} \label{step:add_edge_1}
		\State $result$ $\leftarrow$ False, go to Step~\ref{step:new_return} \label{step:break_1}
	\EndIf
 \EndFor
 \State return $result$\label{step:new_return}
\end{algorithmic}
\end{algorithm}

The pseudocode of the proposed algorithm is given in Algorithm~\ref{alg:k=1}. It takes as input a  bipartite graph $\B_1 =(V_1, \widetilde{V}_1, \E_1)$ and an edge set $S_1 \subseteq \E_1$, and outputs if  $\B_1$ has a perfect matching for every  one edge removal from $S_1$. We define cost function $c_0$ on the edges in $\B_1$ as shown in Step~\ref{step:cost1}. Let $M_{0}$ denotes a minimum cost perfect matching in $\B_1$ under cost function $c_0$ (Step~\ref{step:MCMM}). Note that, if $\B_1 = \B(\bA, \bB, \bC, \bK)$ and $S_1=\E_K$, then the number of feedback edges in $M_0$ is the least number of feedback edges required to have a perfect matching in $\B(\bA, \bB, \bC, \bK)$. Let $F_{0} := M_{0} \cap S_1 $, where $F_0 = \{f_1, \ldots, f_\ell\}$  (Step~\ref{step:F0}). In Step~\ref{step:res_1}, we initialize $result$ variable to `True'. In iteration $i$ of the {\bf for} loop,  the algorithm checks if there exists a perfect matching in  $\B_1$ after removing edge $f_i$ (Step~\ref{step:remove_fi}). If a matching does not exist, then  it outputs `False' (Step~\ref{step:break_1}). On the other hand, if a matching exists, then the algorithm proceeds with the removal of the next edge in $F_0$. Finally, it returns $result$ as the output.

For a given structured system, the theorem below proves that Algorithm~\ref{alg:k=1} solves Problem~\ref{prob:verification} for $\gamma=1$.  Further, we also give the complexity of Algorithm~\ref{alg:k=1}.
\begin{theorem}\label{th:k=1}
Consider a closed-loop structured system $(\bA, \bB, \bC, \bK)$. Let  Assumption~\ref{asm:irr} holds. Algorithm~\ref{alg:k=1}, which takes as input the bipartite graph $\B_1$ and edge set $S_1$, outputs solution to Problem~\ref{prob:verification} for $\gamma=1$ for $\B_1 = \B(\bA, \bB, \bC, \bK)$ and $S_1=\E_K$. Also, the complexity of Algorithm~\ref{alg:k=1} is $O(n^{3.5})$, where $n$ denotes the number of states in the system.
\end{theorem}
\begin{proof}
Here, $F_0= M_0\cap\E_K$.  This implies $M_0 \cap \{\E_K \setminus F_0\} = \emptyset$.  Thus for any one edge removal from $\E_K \setminus F_0$, the system $(\bA, \bB, \bC, \bK)$ is resilient (since perfect matching $M_0$ exists). Hence it is enough to check if there exists a perfect matching in $\B(\bA, \bB, \bC, \bK)$ for every one edge removal from $F_{0}$. Note that, $|F_0| = \ell \geqslant 1$, otherwise the system is always resilient.  In each iteration of the {\bf for} loop of Algorithm~\ref{alg:k=1}, we remove an edge from $F_0$ and check for the existence of a perfect matching. If a perfect matching does not exist, then the algorithm concludes that the system is not one edge resilient. On the other hand, if there exists a perfect matching for every one edge removal from $F_{0}$, then the algorithm concludes that the system is resilient to any one  edge removal from $F_0$. Hence Algorithm~\ref{alg:k=1} solves Problem~\ref{prob:verification} for $\gamma=1$.

Finding minimum cost perfect matching in $\B(\bA, \bB, \bC, \bK)$ has  complexity $O(n^{3})$ \cite{Die:00}. In every iteration, Algorithm~\ref{alg:k=1} finds a perfect matching. Note that the maximum number of iterations is $\ell$. Further, $m=O(n)$, $p=O(n)$ together implies  $\ell = O(n)$. Finding perfect matching in $\B(\bA, \bB, \bC, \bK)$ has complexity $O(n^{2.5})$ \cite{Die:00}.  All the other steps are of linear complexity. Hence  complexity of Algorithm~\ref{alg:k=1} is $O(n^{3.5})$. 
\end{proof}
\vspace*{-2 mm}
\subsection{Algorithm and Results for $\gamma=2$}\label{subsec:k=2}
In this subsection, we present an algorithm to solve Problem~\ref{prob:verification} for $\gamma = 2$ and then prove its correctness and complexity.
 
\begin{algorithm}[t]
  \caption{Psuedocode to verify resilience of bipartite graph $\B_2$ for  removal of any two edges from set $S_2$
  \label{alg:k=2}}
  \begin{algorithmic}
  	\State \textit {\bf Input:} Bipartite graph $\B_2 = (V_2, \widetilde{V}_2, \E_2)$ and  set $S_2 \subset \E_2$
	\State \textit{\bf Output:} True or False 
  \end{algorithmic}
  \begin{algorithmic}[1]
  	\State Define cost function, $c_2(e) \leftarrow \begin{cases}
  1, {\rm for~} e \in S_2,\\
  0, {\rm otherwise}.
\end{cases}  \label{step:cost2}
  $
  \State Find min-cost perfect matching  in $\B_2$ using $c_2$, say $M_{2}$ \label{step:MCMM2}
  \State Define $F_{2}:= M_2 \cap  S_2$, where $F_{2}=\{f_1,\ldots,f_{\ell_2}\}$\label{step:F0_2}
 \State  Define $result_1$ $\leftarrow$ True and $result_2$ $\leftarrow$ True \label{step:res1res2}
  	\For{$i=1$ to $\ell_2$, $i++$}\label{step:for1}
  		\For{$j=i+1$ to $\ell_2$, $j++$}\label{step:for2}
    		\State Check  for existence of perfect matching in $\B_2$ after removing edges $f_i$ and $f_j$ \label{step:remove_fifj}
    		\If {perfect matching does not exists}\label{step:if_Mij}				\State  $result_1$ $\leftarrow$ False, go to Step~\ref{step:retres1res2}\label{step:res1false}
			\EndIf
		\EndFor\label{step:endfor2}
  	\EndFor\label{step:endfor1}
		\For{$i=1$ to $\ell$, $i++$} \label{step:res1for}
    		\State Define $\B_2^i=(V_2, \widetilde{V}_2, \E_2\setminus \{f_i\})$  \label{step:res1fi} 
			\State Apply Algorithm \ref{alg:k=1} with  $\B_1 =\B_2^i$ and $S_1 = S_2 \setminus F_2$\label{step:res1algo1}
			\State $result_2$ $\leftarrow$ Output of Algorithm \ref{alg:k=1}\label{step:res2algo1res}
			\If {$result_2 = $ False}
				\State Go to Step~\ref{step:retres1res2} \label{step:res2break}
			\EndIf
		\EndFor	\label{step:endres1for}
	\State return True if both $result_1$ and $result_2$ are True and False otherwise.\label{step:retres1res2}
  \end{algorithmic}
\end{algorithm}
The pseudocode of the proposed algorithm is given in Algorithm~\ref{alg:k=2}. It takes as input a  bipartite graph $\B_2 = (V_2, \widetilde{V}_2, \E_2)$ and an edge set $S_2 \subseteq \E_2$, and outputs if  $\B_2$ has a perfect matching for  removal of any two edges from $S_2$.

\noindent{\bf Steps~\ref{step:cost2}-\ref{step:res1res2}}:  We define a cost function $c_2$ on the edges in $\B_2$ as shown in Step~\ref{step:cost2}. Let $M_{2}$ denotes a minimum cost perfect matching in $\B_2$ under  $c_2$. Note that, for $\B_2 = \B(\bA, \bB, \bC, \bK)$ and $S_2=\E_K$, the number of feedback edges in $M_2$ is the least number of feedback edges required to obtain a perfect matching in $\B(\bA, \bB, \bC, \bK)$. Let $F_{2} := M_{2} \cap S_2 $, where $F_2 = \{f_1, \ldots, f_{\ell_2}\}$. We initialize variables $result_1$  and $result_2$ to `True'. 

\noindent{\bf Steps~\ref{step:for1}-\ref{step:endfor1}}: In each iteration, the algorithm checks if there exists a perfect matching in  $\B_2$ after removing two edges, say $f_i$ and $f_j$, from $F_2$. If a matching does not exist after removing a pair $\{f_i, f_j\} \in F_2$, then the algorithm sets $result_1$ to False and goes to Step~\ref{step:retres1res2} and outputs `False'. On the other hand, if a matching exists for removal of each pair of edges from $F_2$, then $result_1$ is True. At the end of Step~\ref{step:endfor1}, if $result_1$ is True, then $\B_2$ has a perfect matching for removal of any two edges from $F_2$ and Algorithm~\ref{alg:k=2} proceeds to Step~\ref{step:res1for}.

\noindent{\bf Steps~\ref{step:res1for}-\ref{step:retres1res2}}: In iteration $i$ of the {\bf for} loop,  the algorithm removes an edge $f_i \in F_2$ from $\B_2$. The modified graph is denoted by $\B_2^i$. Then, we run Algorithm~\ref{alg:k=1} with inputs $\B_2^i$ and $S_2 \setminus F_2$ to check if there exists a perfect matching for a failure of one edge from $F_2$ and one edge from $S_2 \setminus F_2$. If Algorithm~\ref{alg:k=1} outputs `False', then we proceed to Step~\ref{step:retres1res2} of Algorithm~\ref{alg:k=2} and return `False' as the output. On the other hand, if Algorithm~\ref{alg:k=1} outputs `True', for all possible $\B_2^i$'s, then we return `True' as the output of Algorithm~\ref{alg:k=2}.

For a given structured system, the result below proves that Algorithm~\ref{alg:k=2} solves Problem~\ref{prob:verification} for $\gamma=2$.  Further, we also give the complexity of Algorithm~\ref{alg:k=2}.
\begin{theorem}\label{th:k=2}
Consider a closed-loop structured system $(\bA, \bB, \bC, \bK)$. Let  Assumption~\ref{asm:irr} holds.  Algorithm~\ref{alg:k=2}, which takes as input a bipartite graph $\B_2$ and edge set $S_2$, outputs solution to Problem~\ref{prob:verification} for $\gamma=2$ for $\B_2 = \B(\bA, \bB, \bC, \bK)$ and $S_2=\E_K$. Also, the complexity of Algorithm~\ref{alg:k=2} is $O(n^{4.5})$, where $n$ denotes the number of states in the system.
\end{theorem}
\begin{proof}
The structured system $(\bA, \bB, \bC, \bK)$ is resilient for $\gamma=2$, if $\B(\bA, \bB, \bC, \bK)$ has a perfect matching for removal of any two edges from $\E_K$. With inputs $\B(\bA, \bB, \bC, \bK)$ and $\E_K$ to Algorithm~\ref{alg:k=2}, $M_2$ is a perfect matching in $\B(\bA, \bB, \bC, \bK)$ and $F_2 \subseteq \E_K$. Then, removal of any two feedback edges can be done in the following ways: (a) both edges from set $F_2$, (b) one edge from $F_2$ and the another from $\E_K \setminus F_2$, and (c) both edges from set $\E_K \setminus F_2$. For the system $(\bA, \bB, \bC, \bK)$ to be resilient for $\gamma=2$, $\B(\bA, \bB, \bC, \bK)$ must have perfect matching for all these cases.  

\noindent {\bf Case~(a)}: For $\B_2 = \B(\bA, \bB, \bC, \bK)$ and $S_2=\E_K$, Steps~\ref{step:for1}-\ref{step:endfor1} of Algorithm~\ref{alg:k=2} checks for existence of perfect matching in $\B(\bA, \bB, \bC, \bK)$ for removal of every pair of edges from $F_2$.  
If there exists no perfect matching after removing some pair of edges in $F_2$, then the algorithm sets $result_1$  to `False' and concludes in Step~\ref{step:retres1res2} that the system is not resilient. 

\noindent {\bf Case~(b)}: For $\B_2 = \B(\bA, \bB, \bC, \bK)$ and $S_2=\E_K$, Steps~\ref{step:res1for}-\ref{step:retres1res2} of Algorithm~\ref{alg:k=2} checks if the system $(\bA, \bB, \bC, \bK)$ is resilient to removal of any edge $f_i \in F_2$ and the other edge from $\E_K \setminus F_2$. Recall that here $\B_2^i$ is obtained after removing edge $f_i$ from $\B(\bA, \bB, \bC, \bK)$. If for any $\B_2^i$  there does not exist a perfect matching, then $result_2$ is set to False. Thus Algorithm~\ref{alg:k=2} concludes if the system is resilient to any two edge removal for case~(b) precisely.

\noindent {\bf Case~(c)}: As $M_2 \cap \{\E_K \setminus F_2\} = \emptyset $, for removal of any two edges  from $\E_K \setminus F_2$, matching $M_2$ exists. Hence the system is resilient to the removal of any two edges for this case. So,  case~(c) requires no verification.

 In the end, the algorithm outputs `True'  if both $result_1$ and $result_2$ are True, i.e., returns True if $\B(\bA, \bB, \bC, \bK)$ has perfect matching for both case~(a) and case~(b). If there exists no perfect matching in either case~(a) or in case~(b), then the output of Algorithm~\ref{alg:k=2} is `False'. This proves the correctness of Algorithm~\ref{alg:k=2} to solve Problem~\ref{prob:verification} for $\gamma=2$.

Finding minimum cost perfect matching in $\B(\bA, \bB, \bC, \bK)$ has  complexity $O(n^{3})$. Here, $\ell_2= O(n)$. For case~(a) there are $O(n^2)$ possible pair of edges and for every possible pair of edge removal, we check if there exists a perfect matching. The time complexity to find perfect matching is $O(n^{2.5})$. So the time complexity for case~(a) is $O(n^{4.5})$.
In case~(b), for every edge removal from $F_2$ we apply Algorithm~\ref{alg:k=1}. The time complexity of Algorithm~\ref{alg:k=1} is $O(n^{3.5})$ and $|F_2|=O(n)$. So the time complexity for case~(b) is $O(n^{4.5})$. All other steps of the algorithm are of linear complexity. Hence the total time complexity of Algorithm~\ref{alg:k=2} is $O(n^{4.5})$. 
\end{proof}

\vspace*{-2 mm}
\subsection{Algorithm and Results for General $\gamma$}\label{subsec:k=k}
In this subsection, we present the key steps of the recursive algorithm to solve Problem~\ref{prob:verification} for general $\gamma$. The inputs to the algorithm are bipartite graph $\B_\gamma = (V_\gamma, \widetilde{V}_\gamma, \E_\gamma)$ and edge set $S_\gamma \subseteq \E_\gamma$. 

Consider a bipartite graph $\B_\gamma$ and an edge set $S_k$. Find a perfect matching in $\B_\gamma$, say $M_\gamma$, that consists of minimum number of edges from $S_\gamma$ (using a min-cost perfect matching algorithm with non-zero uniform cost on edges in $S_\gamma$ and $0$ cost on other edges). Define $F_\gamma := M_\gamma \cap S_\gamma$ and let $|F_\gamma|=\ell_\gamma$. Removal of $\gamma$ edges from $S_\gamma$ can be done in $(\gamma+1)$  ways, case~$(0), \ldots,$ case~$(\gamma)$, where for case~($q$) we consider  removal of $(\gamma-q)$ edges from $F_\gamma$ and remaining $q$ edges from $S_\gamma \setminus F_\gamma$.

\noindent {\bf Step~1}: For case~(0), check the existence of perfect matching for removal of every possible $\gamma$ number of edges from $F_\gamma$.

\noindent {\bf Step~2}: For case~(1), remove $(\gamma-1)$ number of edges from $F_\gamma$ and apply  Algorithm~\ref{alg:k=1} with $\B_1$ as the modified bipartite graph obtained after removing $(\gamma-1)$ edges from $\B_\gamma$ and set $S_1$ as $S_\gamma \setminus F_\gamma$. This is done for every possible $(\gamma-1)$ edges.

\noindent {\bf Step~3}: For case~(2), remove $(\gamma-2)$ number of edges from $F_\gamma$ and apply  Algorithm~\ref{alg:k=2} with $\B_2$ as the modified bipartite graph obtained after removing $(\gamma-2)$ edges from $\B_\gamma$ and set $S_2$ as $S_\gamma \setminus F_\gamma$. This is done for every possible $(\gamma-2)$ edges.

\noindent {\bf Step~4}: We follow the similar lines for case~(3) to the case~$(\gamma-1)$. In case~$(q)$, we remove  $(\gamma-q)$ number of edges from $F_\gamma$ and apply the algorithm for $\gamma=q$ with $\B_q$ as the modified bipartite graph obtained after removing $(\gamma-q)$ edges from $\B_\gamma$ and set $S_q$ as $S_\gamma \setminus F_\gamma$. This is done for every possible $(\gamma-q)$ edges.

\noindent {\bf Step~5}: In case~($\gamma$), as $M_\gamma \cap \{S_\gamma \setminus F_\gamma\} = \emptyset $, for any $\gamma$ edge removal from $S_\gamma \setminus F_\gamma$ matching $M_\gamma$ exists. 

For $\B_\gamma=\B(\bA, \bB, \bC, \bK)$ and $S_\gamma = \E_K$, the above steps solve Problem~\ref{prob:verification} for general $\gamma$. The proof of this claim and computational complexity involved is given  in Theorem~\ref{th:algo_k}.

\begin{theorem}\label{th:algo_k}
Consider a closed-loop structured system $(\bA, \bB, \bC, \bK)$. Let Assumption~\ref{asm:irr} holds. Then, Steps~1-5 solves Problem~\ref{prob:verification} when  $\B_k = \B(\bA, \bB, \bC, \bK)$ and edge set $S_k=\E_K$. Further, the time complexity is $O(2^{\gamma-1}n^{\gamma+2.5})$, where $n$ denotes the number of states in the system.
\end{theorem}
\begin{proof}
For $\B_\gamma = \B(\bA, \bB, \bC, \bK)$ and $S_\gamma = \E_K$, $M_\gamma$ is a perfect matching in $\B(\bA, \bB, \bC, \bK)$ and $F_\gamma \subseteq \E_K$. Removal of any $\gamma$ edges from $\E_K$ can be done in $(\gamma+1)$  ways, case~$(0), \ldots,$ case~$(\gamma)$, where for case~($q$) we consider  removal of $(\gamma-q)$ edges from $F_\gamma$ and remaining $q$ edges from $\E_K \setminus F_\gamma$. For the system to be resilient, it must have a perfect matching in all of these cases.
Steps~1-5 checks for the existence of perfect matching for  each of these cases and  hence verifies resilience of the system for removal of any $\gamma$ feedback edges correctly.

 Now we prove the complexity of the algorithm.
Let the theorem statement be denoted as $P(\gamma)$. We prove the theorem using Strong Induction. \\
\noindent Base step: For $\gamma=1$, $P(1)$ is the complexity of Algorithm~\ref{alg:k=1}, which is $O(n^{3.5})$ (Theorem~\ref{th:k=1}). Hence $P(\gamma)$ is true for $\gamma=1$. \\
\noindent Induction step: Assume that the statement $P(\gamma)$ is true for $\gamma \in \{1,2,\ldots,q\}$. Now we have to prove that the statement $P(\gamma)$ is true for $\gamma=q+1$. 
Consider removal of $q+1$ links. This can be done in $q+2$ possible cases as shown in Steps~1-5.

 In case~($0$), we check the existence of perfect matching for removal of every possible $(q+1)$ number of edges from $F_{q+1}$. As $\ell_{q+1}$ is $O(n)$, there are $\binom{n}{q+1}$ possible combinations which is $O(n^{q+1})$ and time complexity to check  the existence of perfect matching for each combination is $O(n^{2.5})$, so the time complexity of case~($0$) is $O(n^{q+3.5})$. 

 In case ($i$), where $i \in \{1,\ldots,q\}$, for removal of every possible $(q+1-i)$ number of edges from $F_{q+1}$ we apply the algorithm for $\gamma=i$ with $\B_i$ as the modified bipartite graph obtained after removing $(q+1-i)$ edges from $\B_{q+1}$ and set $S_i$ as $S_{q+1} \setminus F_{q+1}$. From the strong induction hypothesis, the time complexity of algorithm for $\gamma=i$ is $O(2^{i-1}n^{i+2.5})$. As $\ell_i$ is $O(n)$, the possible combinations for removing $(q+1-i)$ edges from $F_{q+1}$ is $O(n^{q+1-i})$ and hence the total time complexity of case ($i$) is $O(2^{i-1}n^{q+3.5})$.

Case~($q+1$) requires no verification. In addition to these a minimum cost perfect matching algorithm is used which is of complexity $O(n^3)$. Thus, the total time complexity of the algorithm for $\gamma=q+1$ is 

\begin{eqnarray}
O(n^{q+3.5})& + & \Sigma^{q}_{i=1} O(2^{i-1}n^{q+3.5})  + O(n^3) \nonumber \\
& =& O(2^{q}n^{q+3.5}) 
 =  O(2^{(q+1)-1}n^{(q+1)+2.5}) \nonumber
\end{eqnarray}

So the statement $P(\gamma)$ is true for $\gamma=q+1$, whenever it is true for $\gamma$ $\in$ $\{1, \ldots, q\}$. Hence by the principle of mathematical induction, the statement $P(\gamma)$ is true for all $\gamma$. This concludes the proof. 
\end{proof}

Note that, solving Problem~\ref{prob:verification} using our approach has complexity polynomial in $n$ and exponential in $\gamma$, where $\gamma << n$.  The value of $\gamma$ is typically small as the attacker will attack fewest number of links to disable the operation of the system on account of the resource and the infrastructure constraints involved in an attack.
\begin{remark}\label{rem:comp2}
For a fixed $\gamma$, $2^{\gamma-1}$ is not varying with $n$. So the time complexity of the algorithm for general $\gamma$ can be written as $O(n^{\gamma+2.5})$.  
\end{remark}
Solving Problem~\ref{prob:verification} using an exhaustive search-based technique involves checking existence of perfect matching for $\binom{n^2}{\gamma}$ number of bipartite graphs. Thus the computational complexity of an exhaustive search-based technique is $O(n^{2\gamma+2.5}) = O(n^\gamma \,n^{\gamma+2.5})$. Our algorithm is computationally  efficient than applying a brute-force exhaustive search-based algorithm. 
When the links under attack are small, the computational efficiency of our approach is  better.


\section{Conclusion}\label{sec:conclu}
This paper considered the resilience of a large scale closed-loop structured system  when subjected to dysfunctional feedback connections. We discussed two problems in this paper: (i)~given a structured system, input, output, and feedback matrices, verify whether the closed-loop system retains the arbitrary  pole placement property even after the simultaneous failure of any subset of feedback links of cardinality at most $\gamma$ ({\em verification problem}) and (ii)~given a structured system, input and output matrices, design a sparsest feedback matrix such that the resulting closed-loop system retains the arbitrary  pole placement property even after the simultaneous failure of any subset of feedback links of cardinality at most $\gamma$ ({\em design problem}). 

Firstly, we showed that the verification problem is NP-complete (Theorem~\ref{th:NP}). The complexity of the problem is obtained from a reduction of a known NP-complete problem, the blocker problem. We also showed that the verification problem is NP-complete even when the state digraph of the structured system is irreducible. 
Subsequently, we  proposed an algorithm for solving Problem~\ref{prob:verification} for the class of irreducible structured systems. We first proposed  polynomial time algorithms to solve the resilience problem for the case of one edge removal and two edges removal, i.e., $\gamma=1, 2$, respectively (Algorithms~\ref{alg:k=1}, \ref{alg:k=2}). The correctness and complexity of Algorithms~\ref{alg:k=1}, \ref{alg:k=2} are also proved in the paper (Theorems~\ref{th:k=1}, \ref{th:k=2}). Finally,  we considered the general  case, where $\gamma$ is any positive number. For the general case, we proposed a recursive algorithm which is {\em pseudo-polynomial} in factor $\gamma$ and proved the correctness (Theorem~\ref{th:algo_k}). We show that our algorithm performs much better than an exhaustive search-based algorithm in general and specifically for smaller values of $\gamma$.

We proved that the sparsest resilient feedback design problem is NP-hard (Theorem~\ref{th:prob1_NP}). The NP-hardness of the problem is obtained from a reduction of a known NP-complete problem, the minimum set multi-covering problem. We also proved that the design problem is inapproximable to factor $(1-o(1))\log\,n$, where $n$ denotes the system dimension (Theorem~\ref{th:approxi_NP_sol}). We showed that the design problem is NP-hard for two special cases as well; when the state digraph of the structured system is irreducible and also when all the state nodes in the state digraph are spanned by disjoint cycles (structurally cyclic). We then analyzed structurally cyclic systems with a special feedback structure, so-called back-edge feedback structure, for which the  NP-hardness  and the inapproximability  results hold.  A polynomial-time  $O(\log\,n)$-optimal approximation algorithm to solve the design problem for structurally cyclic systems with back-edge feedback structure is presented by reducing the design problem to minimum set multi-covering problem (Theorem~\ref{th:th_multi_approx}). Identifying other relevant feedback structures  that have computationally efficient solution approach for the verification and the design problem is part of future work. Improving the complexity of the pseudo-polynomial algorithm for solving the verification problem is also part of future work.
\bibliographystyle{myIEEEtran}  
\bibliography{Resilience}
\end{document}